\documentclass[10pt, english,righttag, intlim]{amsart}
\usepackage{amsmath, amssymb, amsthm, fancyhdr, verbatim, graphicx}
\usepackage{enumerate}
\usepackage{enumitem}
\usepackage[all]{xy}
\usepackage{accents}
\usepackage{mathtools}
\usepackage{bbm} 
\usepackage{extarrows}
\usepackage{stmaryrd}
\usepackage{array}
\usepackage{float}
\usepackage[OT2, T1]{fontenc}
\usepackage[unicode=true]{hyperref}
\usepackage{framed}
\usepackage{url}
\usepackage{scalerel}
\usepackage{booktabs}
\usepackage{bbding}
\usepackage{pifont}

\usepackage[usenames,dvipsnames]{xcolor}
\usepackage{mathrsfs}
\usepackage{caption}
\usepackage{tikz-cd}

\definecolor{DresdenDarkGreen}{HTML}{576A4F}

\definecolor{RedOcher}{HTML}{8f5b57}

\definecolor{Azurite}{HTML}{274953}

\newcommand{\R}{\mathbb{R}}
\renewcommand{\C}{\mathbb{C}}
\newcommand{\G}{\mathbb{G}}
\newcommand{\Q}{\mathbb{Q}}
\newcommand{\Z}{\mathbb{Z}}
\renewcommand{\P}{\mathbb{P}}

\newcommand{\B}{\mathbb{B}}
\newcommand{\N}{\mathbb{N}}

\newcommand{\GG}{\mathcal{G}}
\newcommand{\OO}{\mathcal{O}}
\newcommand{\HH}{\mathcal{H}}

\newcommand{\g}{\mathfrak{g}}

\newcommand{\fa}{\mathfrak{a}}

\newcommand{\fg}{\mathfrak{g}}

\newcommand{\fH}{\mathfrak{H}}

\newcommand{\gl}{\mathfrak{gl}}
\renewcommand{\sl}{\mathfrak{sl}}

\newcommand{\sG}{\mathscr{G}}
\newcommand{\sH}{\mathscr{H}}

\newcommand{\sM}{\mathscr{M}}

\newcommand{\cA}{\mathcal{A}}

\newcommand{\cG}{\mathcal{G}}
\newcommand{\cH}{\mathcal{H}}
\newcommand{\cI}{\mathcal{I}}

\newcommand{\cM}{\mathcal{M}}
\newcommand{\cN}{\mathcal{N}}

\newcommand{\cP}{\mathcal{P}}

\newcommand{\cU}{\mathcal{U}}
\newcommand{\cV}{\mathcal{V}}
\newcommand{\cW}{\mathcal{W}}

\newcommand{\Qp}{\mathbb{Q}_p}
\newcommand{\Ql}{\mathbb{Q}_{\ell}}
\newcommand{\Zp}{\mathbb{Z}_{p}}
\newcommand{\Zl}{\mathbb{Z}_{\ell}}

\newcommand{\Qpbar}{\ol{\mathbb{Q}}_{p}}

\newcommand{\wt}[1]{\widetilde{#1}}

\newcommand{\Gal}{\operatorname{Gal}}

\newcommand{\ul}[1]{\underline{#1}}
\newcommand{\ol}[1]{\overline{#1}}
\newcommand{\wh}[1]{\widehat{#1}}

\DeclareMathOperator{\proet}{pro\acute{e}t}

\makeatletter
\newcommand*\bigcdot{\mathpalette\bigcdot@{.7}}
\newcommand*\bigcdot@[2]{\mathbin{\vcenter{\hbox{\scalebox{#2}{$\m@th#1\bullet$}}}}}
\makeatother

\def\ocirc#1{\accentset{\circ}{#1}}
\newcommand{\modifier}[1]{}

\newcommand{\bs}{\backslash}
\newcommand{\surj}{\twoheadrightarrow}

\newcommand{\dirlim}{\varinjlim}
\newcommand{\invlim}{\varprojlim}

\newlist{renumerate}{enumerate}{3}
\setlist[renumerate]{label=(\roman*),before=\raggedright, itemsep=0mm}
\newlist{arenumerate}{enumerate}{3}
\setlist[arenumerate]{label=(\arabic*),before=\raggedright, itemsep=0mm}
\newlist{aenumerate}{enumerate}{3}
\setlist[aenumerate]{label=(\alph*),before=\raggedright, itemsep=0mm}
\setlist[enumerate]{itemsep=0mm}

\DeclareMathOperator{\GL}{GL}
\DeclareMathOperator{\SL}{SL}

\DeclareMathOperator{\PGL}{PGL}
\DeclareMathOperator{\Sp}{Sp}

\DeclareMathOperator{\Hom}{Hom}

\DeclareMathOperator{\Ind}{Ind}

\DeclareMathOperator{\Aut}{Aut}
\DeclareMathOperator{\Rep}{Rep}

\DeclareMathOperator{\Spec}{Spec\,}

\DeclareMathOperator{\Lie}{Lie}
\DeclareMathOperator{\LT}{LT}

\DeclareMathOperator{\Res}{Res}
\DeclareMathOperator{\res}{res}

\DeclareMathOperator{\Ext}{Ext}

\DeclareMathOperator{\id}{id}

\DeclareMathOperator{\Ad}{Ad}

\DeclareMathOperator{\dR}{dR}

\DeclareMathOperator{\cores}{cor}

\DeclareMathOperator{\Spf}{Spf}

\DeclareMathOperator{\cts}{cts}

\DeclareMathOperator{\Spa}{Spa}

\DeclareMathOperator{\SU}{SU}

\DeclareMathOperator{\cInd}{c-Ind}

\DeclareMathOperator{\Ab}{\mathbf{Ab}}

\DeclareMathOperator{\Set}{\mathbf{Set}}
\DeclareMathOperator{\Grp}{\mathbf{Grp}}
\DeclareMathOperator{\Grpd}{\mathbf{Grpd}}
\DeclareMathOperator{\Ring}{\mathbf{Ring}}

\DeclareMathOperator{\Top}{\mathbf{Top}}
\DeclareMathOperator{\Mod}{Mod}

\DeclareMathOperator{\Ch}{\mathbf{Ch}}

\DeclareMathOperator{\Cond}{\mathbf{Cond}}
\DeclareMathOperator{\cond}{cond}

\DeclareMathOperator{\Solid}{\mathbf{Solid}}
\DeclareMathOperator{\solid}{solid}

\usetikzlibrary{calc,decorations.pathmorphing,shapes}
\newcounter{sarrow}

\tikzcdset{scale cd/.style={every label/.append style={scale=#1},
    cells={nodes={scale=#1}}}}

\newtheorem{thm}{Theorem}[section]
\newtheorem{cor}[thm]{Corollary}
\newtheorem{prop}[thm]{Proposition}
\newtheorem{lem}[thm]{Lemma}

\theoremstyle{definition}
\newtheorem{defn}[thm]{Definition}
\newtheorem{defns}[thm]{Definitions}

\newtheorem{exmp}[thm]{Example}

\theoremstyle{remark}
\newtheorem{rem}[thm]{Remark}

\usepackage[letterpaper,left=1in,right=1in,top=1.2in,bottom=1.2in]{geometry}

\usepackage{palatino}

\DeclareFontFamily{OT1}{rsfs}{}
\DeclareFontShape{OT1}{rsfs}{n}{it}{<-> rsfs10}{}
\DeclareMathAlphabet{\mathscr}{OT1}{rsfs}{n}{it}

\numberwithin{equation}{section}
\usepackage[cal=euler]{mathalfa}

\hypersetup{
  colorlinks   = true, 
  urlcolor     = DresdenDarkGreen, 
  linkcolor    = Azurite, 
  citecolor    = RedOcher 
}

\theoremstyle{definition}

\usepackage[
backend=biber,
sorting=nyvt,
style=alphabetic,
maxbibnames=5,
backref=true
]{biblatex}
\renewbibmacro{in:}{}
\DeclareFieldFormat[article]{volume}{\textbf{#1}}
\DeclareFieldFormat[article]{number}{no.~#1}
\DeclareFieldFormat[article]{year}{\mkbibparens{#1}}
\renewbibmacro*{volume+number+eid}{%
  \printfield{volume}%
  \setunit*{\addspace}
  \printfield{number}%
  \setunit{\addcomma\space}
  \printfield{eid}%
}

\DeclareMathOperator{\Dr}{Dr}

\DeclareMathOperator{\Adic}{Adic}
\DeclareMathOperator{\Cont}{Cont}
\DeclareMathOperator{\EDis}{EDis}
\newcommand{\Ksolid}{K_{\square}}
\theoremstyle{theorem}
\newtheorem{Th}{Theorem}

\tikzcdset{
  cells={font=\everymath\expandafter{\the\everymath\displaystyle}},
}

\title{On the pro-\'etale cohomology of quotient stacks of Drinfeld spaces}
\author{Zecheng Yi}
\address{Department of Mathematics and Statistics\\Boston University\\Boston, MA 02215, USA}
\email{\url{zechengyi0822@gmail.com}}

\begin{document}
\maketitle
\begin{abstract}
	Let  $\cH^{n-1}_{K}$ denote the $(n-1)$-dimensional Drinfeld space over a $p$-adic field $K$.  We give an explicit description of the $\ell$-adic and $p$-adic pro-\'etale cohomology of quotient stacks $[\cH^{n-1}_{K}/\GL_n(\OO_K)]$ and $[\cH^{n-1}_{K}/\GL_n(K)]$, which are moduli stacks of special formal $\OO_D$-modules. The computation makes use of  the isomorphism between the Lubin-Tate tower and the Drinfeld tower due to Faltings and Scholze--Weinstein, as well as  the $p$-adic pro-\'etale cohomology of the Drinfeld spaces computed by Colmez--Dospinescu--Nizio{\l}.

	As an application, we also compute the continuous group cohomology of $\GL_n(\Qp)$ over duals of generalized  Steinberg representations over $\Qp$. 
\end{abstract}
\tableofcontents

\section{Introduction}
Let $K$ be a finite extension of $\Qp$ and let $C$ be the completion of an algebraic closure of $K$. The cohomology of Drinfeld spaces  was first studied by Schneider--Stuhler \cite{SS1991-1} for any cohomology theory satisfying certain axioms. Examples of such cohomology theories include de Rham cohomology and $\ell$-adic ($\ell\neq p$) \'etale cohomology. 

The $(n-1)$-dimensional  Drinfeld space $\cH^{n-1}_{K}$ admits an action of $\GL_n(K)$ and carries an interesting moduli interpretation. Specifically, if we consider the quotient stack $[\cH^{n-1}_{K}/\GL_n(\OO_K)]$ associated to $\cH^{n-1}_{K}$, we get the moduli stack of isomorphism classes of special formal $\OO_D$-modules (as defined in \cite{Drinfeld1976-1}) of height $n^2$; similarly, the quotient stack $[\cH^{n-1}_{K}/\GL_n(K)]$ gives rise to the moduli stack of height $n^2$ special formal $\OO_D$-modules up to quasi-isogeny. 

In this article, we compute the $\ell$-adic and $p$-adic pro-\'etale cohomology of these two moduli stacks. For the $\ell$-adic case, since the $\ell$-adic pro-\'etale cohomology satisfies the axioms listed in \cite{SS1991-1}, we have
\begin{equation}\label{1-l}
	H^{r}_{\proet}(\cH^{n-1}_{C},\Ql(r))\cong \Sp_r(\Ql)^{\ast}
\end{equation}
for $0\leq r\leq n-1$ where $\Sp_r(\Ql)^{\ast}$ is the $\Ql$-dual of the generalized Steinberg representation $\Sp_r(\Ql)$ of $\GL_n(K)$. By computing continuous group cohomology of $\GL_n(\OO_K)$ and $\GL_n(K)$ over $\Sp_r(\Ql)^{\ast}$, we get: 
\begin{Th}[Theorem \ref{thm-l-iso} \& Theorem \ref{thm-l-qiso}]\label{thm-0}~{}
\begin{arenumerate}
	\item The cohomology group $H^{r}_{\proet}([\cH^{n-1}_{K}/\GL_n(\OO_K)], \Ql)$ is isomorphic to $\Ql$ when $r=0$ or 1, and vanishes in all other degrees. 
	\item The cohomology group $H^{r}_{\proet}([\cH^{n-1}_{K}/\GL_n(K)],\Ql)$ is isomorphic to $\Ql$ when $r=0$ or 2, isomorphic to $\Ql^2$ when $r=1$, and vanishes in all other degrees.  
\end{arenumerate}
\end{Th}

Our strategy to show Theorem \ref{thm-0} is fairly straightforward. First we descend (\ref{1-l}) down to the field $K$. The proof then rests on the aforementioned results of group cohomology together with the use of the Hochschild-Serre spectral sequence for pro-\'etale cohomology.

As for the $p$-adic case, the computation is  more involved. In \cite{CDN2020-1}, Colmez--Dospinescu--Nizio{\l} described the $p$-adic pro-\'etale cohomology of $\cH^{n-1}_{C}$ using a short exact sequence
\[
0\rightarrow \Omega ^{r-1}(\cH^{n-1}_{C})/\ker d \rightarrow H^{r}_{\proet}(\cH^{n-1}_{C}, \Qp(r))\rightarrow \Sp_r(\Qp)^{\ast}\rightarrow 0.
\] 
It is then very tempting to mimic what we did in the $\ell$-adic case. However, although the differential part $\Omega ^{r-1}(\cH^{n-1}_{C})/\ker d$ wouldn't survive  Galois descent, it is still  difficult to compute the group cohomology $H^{\ast}_{\cts}(\GL_n(K), \Sp_r(\Qp)^{\ast})$ and to control the spectral sequence converging to $H^{\ast}_{\proet}([\cH^{n-1}_{K}/\GL_n(K)], \Qp)$.  

To salvage this, we use the isomorphism between the Drinfeld tower and the Lubin-Tate tower from \cite{SW2013-1}. Specifically, let $\breve{K}$ denote the completion of the maximal unramified extension of $K$, \cite[Theorem~E]{SW2013-1} allows us to trade the action of $\GL_n(\OO_K)$ on $\cH^{n-1}_{K}$ (resp. $\GL_n(K)$ on $\cH^{n-1}_{K}$) for the action of the Morava stabilizer group $\G_n$ on the Lubin-Tate space  $\LT_{n,\breve{K}}$ (resp. $\G^{0}_{n}:=\G_n\times \Z$ on $\P^{n-1}_{\breve{K}}$), and gives us two isomorphisms of stacks:
\[
	[\cH^{n-1}_{K}/\GL_n(\OO_K)]\cong [\LT_{n,\breve{K}}/\G_n]\quad\text{and}\quad	[\cH^{n-1}_{K}/\GL_n(K)]\cong [\P^{n-1}_{\breve{K}}/\G_n^{0}].
\]
By transferring to the Lubin-Tate side, our problem is reduced to computing the continuous group cohomology of $\G_n$ over the  pro-\'etale cohomology of $\LT_{n,\breve{K}}$. Let $\Lambda _{\Qp}(x_1,x_3,\cdots,x_{2n-1})$ denote the exterior algebra over $\Qp$ generated by $x_i$ in degree $i$ and let   $\Lambda _{\Qp}(y)$ denote the exterior algebra generated by a single element $y$ in degree 1. Our main result in the $p$-adic case is the following:

\begin{Th}[Theorem \ref{thm-p-iso} \& Theorem \ref{thm-p-qiso}]\label{thm-2}
	Suppose $K/\Qp$ is a finite extension of degree $d$. 
	\begin{arenumerate}
	\item There is an isomorphism of graded $\Qp$-vector spaces
		\[
			H^{\ast}_{\proet}([\cH^{n-1}_{K}/\GL_n(\OO_K)],\Qp)\cong \Lambda _{\Qp}(x_1,x_3,\cdots,x_{2n-1})^{\otimes d}\otimes H^{\ast}_{\cts}(\Gal(\ol{K}/K), \Qp).
		\] 
	\item There is an isomorphism of graded $\Qp$-vector spaces
		\[
			H^{\ast}_{\proet}([\cH^{n-1}_{K}/\GL_n(K)],\Qp)\cong H^{\ast}_{\proet}(\P^{n-1}_{K},\Qp)\otimes \Lambda _{\Qp}(x_1,x_3,\cdots,x_{2n-1})^{\otimes d}\otimes \Lambda _{\Qp}(y),
		\] 
			where 
		\[
			H^{\ast}_{\proet}(\P^{n-1}_{K},\Qp)\cong \begin{cases}
			\Qp  &  \text{if $\ast=0$}\\
			\Qp^{d+1}  &  \text{if $\ast=1$}\\
			\Qp^{d}  &  \text{if $3\leq \ast\leq 2n-1$ and $\ast$ is odd}\\
				0 & \text{otherwise.}
			\end{cases}
		\] 
\end{arenumerate}
\end{Th}

Our computations conducted on the Lubin-Tate side  also have some applications back to the Drinfeld side. Using Theorem \ref{thm-2}, we show (Theorem \ref{thm-Steinberg-coh}) that there is an isomorphism $H^{i}_{\cts}(\GL_2(\Qp), \Sp_1(\Qp)^{\ast})\cong H^{i-1}_{\cts}(\GL_2(\Qp), \Qp)$ for every $i\geq 0$. Furthermore, Theorem \ref{thm-2} suggests that the cohomology groups of  $\GL_n(\Qp)$ over the dual Steinberg representation $\Sp_r(\Qp)^{\ast}$ is a degree-$r$ shift of that over $\Qp$. Consider the exterior algebra  $\Lambda_{\Qp}(x,y,x_3,\cdots,x_{2n-1})$  over $\Qp$ with degrees of generators given by $\left|x\right|=\left|y\right|=1$ and $\left|x_i\right|=i$. We can then identify  $H^{\ast}_{\cts}(\GL_n(\Qp), \Qp)$ with this exterior algebra (cf. Proposition \ref{prop-GLQp}). Using techniques from \cite{Orlik2005-1} for computing $\Ext$-groups between Steinberg representations and also the machinery of solid representations from \cite{RJRC2022-1} and \cite{RJRC2025-1}, we show:

\begin{Th}[Theorem \ref{thm-shift}]
There is an isomorphism of graded $\Qp$-vector spaces 
\[
	H^{\ast}_{\cts}(\GL_n(\Qp), \Sp_r(\Qp)^{\ast})\cong \Lambda _{\Qp}(x,y,x_3,\cdots,x_{2n-1})[-r]
\] 
for  every  $0\leq r\leq n-1$. 
\end{Th}

\subsection*{Acknowledgements} 
I am deeply grateful to my advisor, Jared Weinstein, for suggesting this problem and for his constant encouragement and many helpful pieces of advice. I  thank  Jiawei An, Gabriel Dospinescu, Alexander Petrov, and Juan Esteban Rodr\'iguez Camargo for several helpful and stimulating discussions on topics related to this project. I also thank Jared Weinstein and H\r{a}vard Damm-Johnsen for many helpful comments and corrections on earlier drafts of this text. 

The author was partially supported by NSF grant DMS-2401472. 

\subsection*{Notations}
We use the following notations throughout the paper, unless otherwise stated. 

We denote  by $K$  a  finite extension of $\Qp$, with ring of integers $\OO_K$, uniformizer $\pi$, and residue field $k$. We let $W=W(\ol{k})$ be the ring of Witt vectors of $\ol{k}$, and let $\breve{K}=W[\frac{1}{p}]$ be the completion of the maximal unramified extension of $K$. We let $C$ be the completion of an algebraic closure $\ol{K}$ of $K$, and the  Galois group $\Gal(\ol{K}/K)$ will be denoted by $\cG_K$. 
We use $\ell$ to denote  a prime different from $p$.

\section{Two pairs of moduli stacks}\label{section-2}
Following \cite[Chapter~3]{RZ} and \cite[Chapitre~II]{Fargues2008-1}, we recall the definitions of Drinfeld tower and Lubin-Tate tower at infinite level, their  related  group actions, and also the moduli interpretation of associated quotient stacks. We will also state the isomorphism due to Scholze--Weinstein between the two towers and derive two pairs of isomorphic moduli stacks.   
\subsection{Moduli stacks arising from the Drinfeld tower}\label{section-2.1}

For $n\geq 2$, the $(n-1)$-dimensional Drinfeld symmetric space over $K$ is defined as 
\[
\cH^{n-1}_{K}:=\P ^{n-1}_{K}\bs \bigcup^{}_{H\in \fH}H
\] 
where $\fH$ is the set of $K$-rational hyperplanes inside $\P^{n-1}_{K}$. The space $\cH^{n-1}_{K}$ is a rigid analytic Stein space and has a natural action of  $\GL_n(K)$ on it. Let $D$ be the unique (up to isomorphism) division algebra over $K$ of invariant $1/n$, let $\OO_D$ be its ring of integers, and denote by $\varpi$ the uniformizer of $D$. As defined in \cite{Drinfeld1976-1}, the level-0 Drinfeld moduli space $\sM^{\Dr}_{0}$ arises as the rigid  fiber of certain deformation space of special formal $\OO_D$-modules of dimension $n$ and height $n^2$. Moreover, Drinfeld also proved in \cite{Drinfeld1976-1} that there is a $\GL_n(K)$-equivariant isomorphism  $\sM^{\Dr}_{0}\cong \coprod_{\Z}\cH^{n-1}_{\breve{K}}$. The space $\sM^{\Dr}_{0}$ admits a tower of finite \'etale $\GL_n(K)$-coverings
\[
	\cdots \rightarrow \sM^{\Dr}_{2}\rightarrow \sM^{\Dr}_{1}\rightarrow \sM^{\Dr}_{0}\cong \coprod_{\Z}\cH^{n-1}_{\breve{K}},
\] 
which is called  the \textit{Drinfeld tower}. 

Fix  $G_0$ to be a special formal $\OO_D$-module over $k$ of dimension  $n$ and height $n^2$. 

\begin{defn}\label{defn-DR}
	Let $(R,R^{+})$ be a complete affinoid $(K, \OO_K)$-algebra. The \textit{Drinfeld tower at infinite level} is a functor  $\sM^{\Dr}_{\infty}$ on complete affinoid $(K,\OO_K)$-algebras  whose set of $(R,R^{+})$-points is the set of quadruples  $(G,\iota, \rho,\eta)$ up to isomorphism, where 
	\begin{itemize}
		\item $G$ is a special formal $\OO_{D}$-module over $R^{+}$,
		\item $\iota: W \rightarrow R^{+}$ is a ring homomorphism,
		\item $\rho: G_0\otimes _{W, \iota} R^{+}/pR^{+} \rightarrow G\otimes _{R^{+}} R^{+}/pR^{+}$ is an $\OO_D$-equivariant quasi-isogeny,
		\item $\eta: \OO_D\xrightarrow{\sim} T_p(G)=\invlim G[p^{n}]$ is an isomorphism of $\OO_D$-modules. 
	\end{itemize}
	Two quadruples  $(G,\iota, \rho, \eta)$ and $(G', \iota', \rho', \eta')$ are  isomorphic if there is an isomorphism $f: G\xrightarrow{\sim} G'$ transferring  one set of data to the other. 
\end{defn}

\begin{rem}
Here we have a slight caveat that a complete affinoid algebra $(R,R^{+})$ might not be sheafy, in which case its adic spectrum is not an adic space in the sense of Huber. However, as discussed in  \cite[Section~2.1]{SW2013-1}, the category of complete affinoid algebras has the sheafy ones as a full subcategory, and the functor $(R,R^{+})\mapsto \Spa(R,R^{+})$ is fully faithful when restricted to this subcategory. One could also restrict to pairs $(R,R^{+})$ where $R$ is perfectoid, which guarantees the sheafiness and leads to the "diamond" version of the moduli space. See Remark \ref{rem-diamond} for more on this. 
\end{rem}

The space $\sM^{\Dr}_{\infty}$ bears an action of $\GL_n(K)\times D^{\times}$. The group $\GL_n(K)$ comes from the automorphism group $\Aut^{0}(G_0)$ of $G_0$ in the isogeny category, and $g\in \GL_n(K)$ acts on $\sM^{\Dr}_{\infty}$ by sending $\rho$ to $\rho \circ  g^{-1}$. An element $d\in \OO^{\times}_{D}$ acts on $\sM^{\Dr}_{\infty}$ by sending $\eta$ to $\eta \circ d^{-1}$. Let $\phi$ denote the quotient map $G\surj G/G[\varpi]$. The action of $\varpi\in D^{\times}$ is defined by $\varpi\cdot (G,\iota, \rho, \eta)=(G/G[\varpi], \iota, \phi\circ \rho, \eta\circ \varpi)$. From the space $\sM^{\Dr}_{\infty}$, quotienting by $\OO^{\times}_{D}$ gives $\sM^{\Dr}_{0}$ and quotienting by $D^{\times}$ places us at the level of $\cH^{n-1}_{\breve{K}}$. Let $\varphi\in \Gal(\breve{K}/K)$ be the lift of Frobenius, then $\varphi$ acts on $\sM^{\Dr}_{\infty}$ by sending $\iota$ to $\iota\circ \varphi$. In the mean time, $\varphi$ also induces an isomorphism $f_{\varphi}: G_0\otimes _{W, \iota\circ \varphi}R^{+}/pR^{+}\xrightarrow{\sim} G_0\otimes _{W,\iota}R^{+}pR^{+}$ and a quasi-isogeny $\rho_{\varphi}:=\rho\circ f_{\varphi}: G_0\otimes _{W, \iota\circ \varphi}R^{+}/pR^{+}\rightarrow H\otimes R^{+}/pR^{+}$. The map $(G, \iota\circ \varphi, \rho, \eta)\mapsto (G, \iota, \rho_{\varphi}, \eta)$ then defines an isomorphism $\varphi^{\ast}\sM^{\Dr}_{\infty}\xrightarrow{\sim} \sM^{\Dr}_{\infty}$ and gives the functor $\sM^{\Dr}_{\infty}$ a Weil descent datum. To help visualize how these spaces in the Drinfeld tower relate to one another, we have the following diagram:
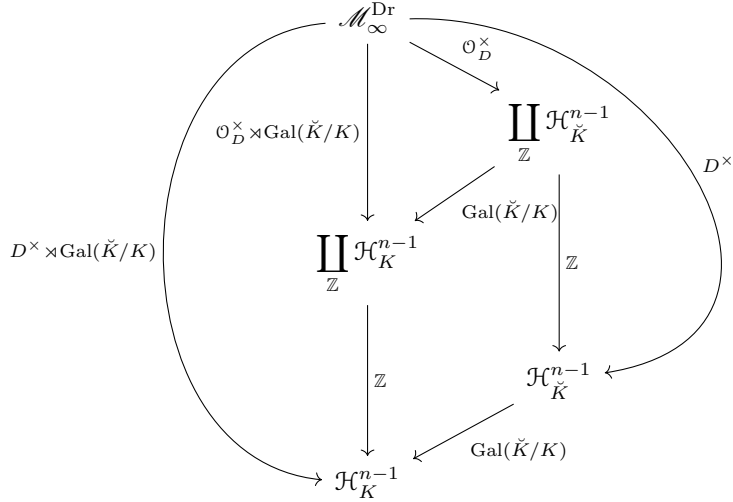
\[
\begin{tikzcd}
	& \sM^{\Dr}_{\infty}\arrow[dd, "\OO_D^{\times}\rtimes \Gal(\breve{K}/K)", swap]\arrow[dr, "\OO_D^{\times}"]\arrow[dddd, bend right = 85, "D^{\times}\rtimes \Gal(\breve{K}/K)", swap, looseness = 1.2]\arrow[dddr, bend left = 63, in = 70,  looseness = 1.6, "D^{\times}"] & \\
	& &\coprod_{\Z}\cH^{n-1}_{\breve{K}}\arrow[dl, "\Gal(\breve{K}/K)"]\arrow[dd, "\Z"]\\
	& \coprod_{\Z}\cH^{n-1}_{K}\arrow[dd, "\Z"]	& \\
	&       &\cH^{n-1}_{\breve{K}}\arrow[dl, "\Gal(\breve{K}/K)"]\\
	& \cH^{n-1}_{K} &
\end{tikzcd}
\]

Building on the moduli interpretation of $\sM^{\Dr}_{\infty}$ , we can similarly regard $\cH^{n-1}_{K}$ as a moduli space. Furthermore, since $\GL_n(K)$ acts on $\cH^{n-1}_{K}$, taking quotients by its subgroups naturally yields some  interesting moduli stacks. Among these, we  are particularly interested in the quotient stacks $[\cH^{n-1}_{K}/\GL_n(\OO_K)]$ and $[\cH^{n-1}_{K}/\GL_n(K)]$. We will explain their moduli interpretation in the following. 

We define a functor $\sG^{+}$ from the category of affinoid adic spaces over $\Spa(K,\OO_K)$ to the category $\Grpd$ of groupoids by assigning to each affinoid adic space $\Spa(R, R^{+})$ the groupoid $\sG^{+}(\Spa(R, R^{+}))$ where: 
\begin{itemize}
	\item Objects in $\sG^{+}(\Spa (R,R^{+}))$ are special formal $\OO_D$-modules  over $R^{+}$ that are quasi-isogenous to the fixed module $G_0$ over $R^+/pR^+$.
	\item Morphisms in $\sG(\Spa(R,R^{+}))$ are isomorphisms between such special formal $\OO_D$-modules.
\end{itemize}
Similarly, we define another  functor $\sG^{0,+}$ from the category of affinoid adic spaces over $\Spa(K,\OO_K)$ to $\Grpd$ by assigning to each affinoid adic space $\Spa (R,R^{+})$ the groupoid $\sG^{0,+}(\Spa (R,R^{+}))$ where:
\begin{itemize}
	\item Objects in $\sG^{0,+}(\Spa (R,R^{+}))$ are special formal $\OO_D$-modules over $R^{+}$ which are quasi-isogenous to $G_0$ over $R^{+}/pR^{+}$.
	\item Morphisms in $\sG^{0,+}(\Spa(R,R^{+}))$ are quasi-isogenies between such special formal $\OO_D$-modules.
\end{itemize}
We let $\sG$ be the sheafification of $\sG^{+}$ and let $\sG^{0}$ be the sheafification of $\sG^{0,+}$, both with respect to the pro-\'etale topology. So $\sG$ and $\sG^{0}$ are functors from the category $\Adic_{K}$ of adic spaces over $\Spa(K,\OO_K)$ to the category $\Grpd$ which pro-\'etale locally on $\Spa(R, R^{+})$ are given by $\sG^{+}$ and $\sG^{0,+}$ respectively.

\begin{prop}\label{prop-Dr-moduli}
	The quotient stacks $[\cH^{n-1}_{K}/\GL_n(\OO_K)]$ and $[\cH^{n-1}_{K}/\GL_n(K)]$ represent the functors $\sG$ and $\sG^{0}$ respectively. 
\end{prop}
\begin{proof}
	From the space $\sM^{\Dr}_{\infty}$, when we quotient out the action of $D^{\times}$, the resulting space $\cH^{n-1}_{\breve{K}}$ only parametrizes triples $(G,\iota, \rho)$. Using the Weil descent datum on $\sM^{\Dr}_{\infty}$, quotienting by $\Gal(\breve{K}/K)$ further forgets the homomorphism $\iota$.  Now  two pairs $(G,\rho)$ and $(G',\rho')$ identify the same point in $\sM^{\Dr}_{\infty}$ if the quasi-isogeny 
	$\rho'\circ \rho^{-1}: G\text{ mod } p \rightarrow G'\text{ mod } p$
lifts to an isomorphism $G\rightarrow G'$ over  $R^{+}$. Thus in the stacky quotient $[\cH^{n-1}_{\breve{K}}/\GL_n(\OO_K)]$, the data of $\rho$ is forgotten and the morphisms between the objects are isomorphisms. When we instead quotient out by the full group $\GL_n(K)$, we are further allowing morphisms over $R^{+}$ which are only invertible upon inverting $p$. Thus  morphisms in $[\cH^{n-1}_{\breve{K}}/\GL_n(K)]$ are  quasi-isogenies.
\end{proof}

\subsection{Moduli stacks arising from the Lubin-Tate tower}
The Lubin-Tate tower consists of spaces parametrizing formal groups. Similar to Section \ref{section-2.1}, we will first recall the construction of these spaces and then explain the moduli interpretation of the associated moduli stacks of our interest. For more details on this, see also \cite[Section~6.4]{SW2013-1} and \cite[Section~3]{BSSW2024-2}.  

Fix $H_0$ to be a 1-dimensional formal group over $k$  of height $n$. The Lubin-Tate space arises as the deformation space of $H_0$. Over a Noetherian local ring of residue characteristic $p$, \cite[Proposition~1]{Tate1967-1} states that the functor $H\mapsto H[p^{\infty}]$ gives an equivalence between the category of $p$-divisible formal groups and the category of connected $p$-divisible groups over this ring. To simplify notation, we use the same symbol $H$ for both a deformation of $H_0$ and its associated $p$-divisible group $H[p^{\infty}]$.

\begin{defn}
	Let $(R,R^{+})$ be a complete affinoid $(K,\OO_K)$-algebra. The \textit{Lubin-Tate tower at infinite level} is a functor  $\sM^{\LT}_{\infty}$ on complete affinoid $(K,\OO_K)$-algebras  whose set of $(R,R^{+})$-points is the set of quadruples $(H,\iota,\rho,\eta)$ up to isomorphism, where 
	\begin{itemize}
		\item $H$ is a 1-dimensional $p$-divisible formal group over $R^{+}$,
		\item $\iota: W\rightarrow R^{+}$ is a ring homomorphism,
		\item $\rho: H_0\otimes _{W,\iota} R^{+}/pR^{+} \rightarrow H\otimes_{R^{+}} R^{+}/pR^{+}$ is a quasi-isogeny,
		\item $\eta: \OO_K^{n}\xrightarrow{\sim} T_p(H)=\invlim H[p^{n}]$ is an isomorphism of Galois modules.
	\end{itemize}
	Two quadruples $(H,\iota,\rho,\eta)$ and $(H,\iota', \rho', \eta')$ are isomorphic if there is an isomorphism $f: H\xrightarrow{\sim} H'$ over $R^{+}$ transferring one set of data to the other. 
\end{defn}

Both $D^{\times}$ and  $\GL_n(K)$ act naturally  on the space $\sM^{\LT}_{\infty}$. 
An element $d\in D^{\times}$ acts on the space by sending $\rho$ to $\rho\circ d^{-1}$. The subgroup $\OO_D^{\times}\subseteq D^{\times}$  arises as the group of automorphisms of the formal group $H_0$, and the group  $D^{\times}$ is the group of automorphisms of $H_0$ in the isogeny category. An element $g\in \GL_n(\OO_K)$ acts on the space via sending $\eta$ to $\eta\circ g$. To extend the action to $\GL_n(K)$, let $H$ be a fixed $p$-divisible group over $R^{+}$ and let $S$ be the set of isomorphism classes of pairs $(H', \phi)$  where $H'$ is another $p$-divisible group over $R^{+}$ and $\phi: H\rightarrow H'$ is a quasi-isogeny over $R^{+}$. This set $S$ is in bijection with the set of Galois stable lattices in $V_p(H)$, cf. \cite[Lemme~II.6.1]{Fargues2008-1}. Now for an element $g\in \GL_n(K)$, $\eta(g\cdot \OO_K^{n})$ is a Galois stable lattice in $V_p(H)$. Thus it gives a pair $(H_g,\phi_g)\in S$. Let $\phi_{g,\ast}: V_p(H)\xrightarrow{\sim} V_p(H_g)$ denote the induced isomorphism on the rational Tate modules. The action of $\GL_n(K)$ is defined via 
\begin{equation}\label{eq-GL-action}
	g\cdot [(H,\iota, \rho,\eta)]=[(H_g,\iota,  (\phi_g\text{ mod } p)\circ \rho,\phi_{g,\ast}\circ \eta\circ g )], \quad g\in \GL_n(K). 
\end{equation}
Let $\LT_n=\Spf W\llbracket u_1,\cdots,u_{n-1}\rrbracket$ be the Lubin-Tate deformation space of $H_0$ as defined in \cite{LT} and let $\LT_{n,\breve{K}}$ denote its generic fiber. We have the following diagram obtained from taking quotients on $\sM^{\LT}_{\infty}$: 
\[
\begin{tikzcd}
	& \sM^{\LT}_{\infty} \arrow[dd] \arrow[dr, "\GL_n(\OO_K)"] \arrow[ddd, "\GL_n(K)", bend right=60, looseness=1.3, swap] & \\
	& & \coprod_{\mathbb{Z}} \LT_{n,\breve{K}} \arrow[dl, "\mathbb{Z}"]\\
	& \LT_{n,\breve{K}} \arrow[d, "\text{Gross-Hopkins}"] &\\
	& \P^{n-1}_{\breve{K}} &
\end{tikzcd}
\]
where the map $\LT_{n,\breve{K}}\rightarrow \P^{n-1}_{\breve{K}}$ is the Gross-Hopkins period map from \cite{GH}.  

As we noted earlier, the group $\OO^{\times}_{D}$ comes from $\Aut_{\ol{k}}(H_0)$. If we treat $H_0$ as a functor on $k$-algebras and also take $\Gal(\ol{k}/k)$ into consideration, we get the \textit{Morava stabilizer group} $\G_n:= \OO^{\times}_{D}\rtimes \hat{\Z}$. 
We use $\G^{0}_{n}:=D^{\times}\rtimes \hat{\Z}$ to denote the analogue of $\G_n$ in the isogeny category. These groups fit into the following diagram: 
\[
\begin{tikzcd}
	& 0\arrow[d] & 0\arrow[d] & &\\
	0\arrow[r] & \OO_D^{\times}\arrow[r]\arrow[d] & \G_n\arrow[r]\arrow[d] & \hat{\Z}\arrow[r]\arrow[d] & 0\\
	0\arrow[r] & D^{\times}\arrow[r]\arrow[d] & \G^{0}_{n}\arrow[r]\arrow[d] & \hat{\Z}\arrow[r] & 0\\
	 & \Z\arrow[d]\arrow[r] & \Z\arrow[d] &  &\\
	 & 0 & 0  & &
\end{tikzcd}
\] 
Let $\varphi\in \Gal(\ol{k}/k)$ be the Frobenius map, then  the action of $\Gal(\ol{k}/k)$ on $\sM^{\LT}_{\infty}$ is given by $\varphi\cdot (H,\iota, \rho, \eta)=(H, \iota\circ \varphi, \rho, \eta)$. Furthermore, $\varphi$ induces an isomorphism $f_{\varphi}: H_0\otimes_{W,\iota\circ \varphi}R^+/pR^+\xrightarrow{\sim} H_0\otimes_{W,\iota}R^+/pR^+$, which then gives a quasi-isogeny $\rho_{\varphi}:=\rho\circ f_{\varphi}: H_0\otimes_{W,\iota\circ \varphi}R^+/pR^+\rightarrow H\otimes R^+/pR^+$. The map $(H, \iota\circ \varphi, \rho, \eta)\mapsto (H, \iota, \rho_{\varphi}, \eta)$ then gives an isomorphism $\varphi^{\ast}\sM^{\LT}_{\infty}\xrightarrow{\sim} \sM^{\LT}_{\infty}$ and equips the functor $\sM^{\LT}_{\infty}$ a Weil descent datum.

The moduli stacks we are interested in here are $[\LT_{n,\breve{K}}/\G_n]$ and $[\P^{n-1}_{\breve{K}}/\G^{0}_{n}]$. To state the moduli interpretation of these two stacks, we first define two functors $\sH^{+}, \sH^{0,+}$ from the category of affinoid adic spaces over $\Spa (K,\OO_K)$ to $\Grpd$ by assigning to $\Spa(R,R^{+})$ the following:
\begin{itemize}
	\item Objects in both $\sH^{+}(\Spa(R,R^{+}))$ and $\sH^{0,+}(\Spa(R,R^{+}))$ are formal groups $H$ over $R^{+}$ which are quasi-isogenous to $H_0$ over $R^{+}/pR^+$.
	\item Morphisms in $\sH^{+}(\Spa(R,R^{+}))$ are isomorphisms between the formal groups.
	\item Morphisms in $\sH^{0,+}(\Spa(R,R^{+}))$ are quasi-isogenies between the formal groups.
\end{itemize}
We let $\sH$ be the sheafification of $\sH^{+}$ and let $\sH^{0}$ be the sheafification of $\sH^{0,+}$, both with respect to the pro-\'etale topology. Thus we have two functors 
\[
	\sH, \sH^{0}: \Adic_K\rightarrow \Grpd
\] 
such that pro-\'etale locally they are defined as $\sH^{+}$ and $\sH^{0,+}$ above respectively.

\begin{prop}\label{prop-LT-moduli}
	The quotient stacks $[\LT_{n,\breve{K}}/\G_n]$ and $[\P^{n-1}_{\breve{K}}/\G_n^{0}]$ represent the functors $\mathscr{H}$ and $\mathscr{H}^{0}$ respectively.
\end{prop}
\begin{proof}

	From the space $\sM^{\LT}_{\infty}$, quotienting out $D^{\times}$ forgets the quasi-isogeny $\rho$. Quotienting  by $\GL_n(\OO_K)$ further forgets the rigidification $\eta$. Using the Weil descent datum of $\sM^{\LT}_{\infty}$, quotienting out by $\hat{\Z}$ forgets the homomorphism $\iota$. The moduli interpretation of $[\LT_{n,\breve{K}}/\G_n]$ then follows. As for the stack $[\P^{n-1}_{\breve{K}}/\G^{0}_{n}]$, notice that the only difference from the earlier quotient is we are now quotienting $\sM^{\LT}_{\infty}$ by $\GL_n(K)$. Let $g\in \GL_n(K)$ and let $(H,\iota, \rho, \eta)$ be a point of $\sM^{\LT}_{\infty}$. Using \cite[Lemme~II.6.1]{Fargues2008-1} and the action of $\GL_n(K)$ described in (\ref{eq-GL-action}), as $\eta(g\cdot \OO_K^{n})$ runs through all  Galois stable lattices in $V_p(H)$, $H_g$ also runs through all formal groups over $R^{+}$ which are quasi-isogenous to $H$. This gives the moduli interpretation of $[\P^{n-1}_{\breve{K}}/\G^{0}_{n}]$.  
\end{proof}

\subsection{The isomorphism between the two towers}
The Drinfeld tower and the Lubin-Tate tower are related through the following celebrated theorem of Scholze-Weinstein:
\begin{thm}[{\cite[Theorem~E]{SW2013-1}}]\label{thm-two-towers}
	The spaces $\sM^{\Dr}_{\infty}$ and $\sM^{\LT}_{\infty}$ are isomorphic as perfectoid spaces. 
\end{thm}

Upon taking quotients of $\sM^{\Dr}_{\infty}\cong \sM^{\LT}_{\infty}$ by subgroups of $\GL_n(K)\times D^{\times}$ and taking descent into consideration, we get:
\begin{cor}\label{cor-iso}
There are isomorphisms of quotient stacks
\[
	[\LT_{n,\breve{K}}/\G_n]\cong [\cH^{n-1}_{K}/\GL_n(\OO_K)]\quad\text{and}\quad	[\P^{n-1}_{\breve{K}}/\G_n^{0}]\cong [\cH^{n-1}_{K}/\GL_n(K)].
\] 
\end{cor}
Based on their moduli interpretation, we will simply refer to the first pair of stacks $[\LT_{n,\breve{K}}/\G_n]$ and $[\cH^{n-1}_{K}/\GL_n(\OO_K)]$ as the \textit{isomorphism stacks}; similarly we  shall call the two stacks in the second pair as \textit{isogeny stacks}. 

\begin{rem}\label{rem-diamond}
	The spaces in the quotient stacks from Corollary \ref{cor-iso} are the generic fibers of their formal models. One can also take the "diamond" generic fibers instead and reinterpret Corollary \ref{cor-iso} as 
	\[
		[\LT^{\diamond}_{n,\breve{K}}/\G_n]\cong [\cH^{n-1,\diamond}_{K}/\GL_n(\OO_K)]\quad\text{and}\quad [\P^{n-1,\diamond}_{\breve{K}}/\G^{0}_{n}]\cong [\cH^{n-1,\diamond}_{K}/\GL_n(K)].
	\] 
	Such stacks have analogous moduli interpretations as given earlier, except restricting the source from the category  $\Adic_K$ to its subcategory of perfectoid spaces over $\Spa(K,\OO_K)$, cf. \cite[Section~3.7]{BSSW2024-2}. 
\end{rem}

\section{Pro-\'etale cohomology of some period domains}\label{section-3}
Our goal of this section is to review the pro-\'etale cohomology, both $\ell$-adic and $p$-adic, of certain rigid spaces. After that, we explain how to equip such cohomology groups with a topology using the theory of condensed mathematics of Clausen--Scholze.  
\subsection{Results of Colmez--Dospinescu--Nizio\l}\label{subsection-3.1}
The study of the cohomology of the Drinfeld spaces traces  back to Schneider and Stuhler \cite{SS1991-1}, where they give a general description of the cohomology groups for any cohomology theory satisfying certain axioms. Although the $p$-adic pro-\'etale cohomology theory does not satisfy the axiom  of  "homotopy invariance" in \cite{SS1991-1}, Colmez--Dospinescu--Nizio{\l} \cite{CDN2020-1} still manage to give a description of $H^{\ast}_{\proet}(\cH ^{n-1}_{C},\Qp)$ in terms of differential forms and the generalized Steinberg representations, which is the key result we aim to highlight here.  Along the way, we will also recall the pro-\'etale cohomology of some other spaces that already  appeared in Section \ref{section-2}.

The pro-\'etale cohomology of projective spaces can be deduced directly from their \'etale cohomology, which is well-understood, see e.g. \cite{Niziol2021-1}.
\begin{thm}\label{thm-projective}
	Let $\ell$ be a prime (possibly equal to $p$) and let $r$ be an integer. When $r$ is even and $0\leq r\leq 2n$, we have
		\[
			H^{r}_{\proet}(\P^{n}_{C}, \Ql)\cong \Ql\left(-\frac{r}{2}\right)
		\] 
	For all other $r$, the cohomology group vanishes. 
\end{thm}

The space $\LT_{n,\breve{K}}$ arises as the generic fiber of the Lubin-Tate space $\LT_n$ and is  isomorphic to the $(n-1)$-dimensional open unit ball $\ocirc{\B}^{n-1}_{\breve{K}}$. Let $\Omega^{r} $ denote the sheaf of differential $r$-forms and  let $d: \Omega^{r}\rightarrow \Omega^{r+1}$ denote the differential map. The cohomology of open unit balls is given as follows.  
\begin{thm}\label{thm-ball}
Let $r\geq 0$ be an integer, we have 
\begin{arenumerate}
	\item $H^{0}_{\proet}(\ocirc{\B}^{n}_{C},\Ql)\cong \Ql$ and all higher cohomology groups vanish, for  $\ell\neq p$.
	\item $H^{0}_{\proet}(\ocirc{\B}^{n}_{C},\Qp)\cong \Qp$ and $H^{r}_{\proet}(\ocirc{\B}^{n}_{C},\Qp(r))\cong \Omega ^{r-1}(\ocirc{\B}^{n}_{C})/\ker d$ for $r\geq 1$. 
\end{arenumerate}
\end{thm}
\begin{proof}
	The $\ell$-adic pro-\'etale cohomology of the open unit ball follows directly from the $\ell$-adic pro-\'etale cohomology of the closed ball. For the $p$-adic case, see \cite[Theorem~3]{CN2020-1}. 
\end{proof}

Before describing the pro-\'etale cohomology of  Drinfeld spaces, we first recall the following class of $\GL_n(K)$-representations; see also \cite[Section~5.2]{CDN2020-1} for more details. Firstly,  let $\Delta =\{1,\cdots,n-1\}$,  which can be identified with the set of simple roots of $\GL_n(K)$. Let $W$ be its Weyl group and let $B\subseteq \GL_n(K)$ be the upper triangular Borel subgroup. For a subset $I\subseteq \Delta $, let $W_I\subseteq W$  be the subgroup generated by permutations $(i,i+1)$ for $i\in I$ and let $P_I=BW_IB$ be the associated parabolic subgroup. Then for any ring $A$ and any subset $I\subseteq \Delta $, the \textit{generalized Steinberg representation} (associated to $I$) is defined as 
\[
	\Sp_I(A)=\frac{C^{\infty}(G/P_I, A)}{\sum^{}_{i\in \Delta \bs I}C^{\infty}(G/P_{I\cup \{i\}}, A)}
\] 
where $C^{\infty}$ denotes the set of smooth functions. For $r\in \{0,\cdots, n-1\}$, we simply write  $\Sp_r(A)$ for the generalized Steinberg representation associated to the subset $\{1,\cdots,n-1-r\}\subseteq \Delta $. Notice that when $r=0$, its associated set is simply $\Delta $ and $\Sp_0(A)$ is the trivial representation over $A$. The ring $A$ we consider will always be a topological ring, and this will give both $\Sp_I(A)$ and its dual $\Sp_I(A)^{\ast}$ a topology.  As discussed in \cite[Section~5.2.2]{CDN2020-1}, when $A=\Qp$, $\Sp_I(A)$ is an $\operatorname{LF}$-space and $\Sp_I(A)^{\ast}$ is a Fr\'echet space. 

Now for $\ell\neq p$, the following theorem can be directly deduced from \cite{SS1991-1} as the $\ell$-adic pro-\'etale cohomology theory satisfies the axioms listed there.  
\begin{thm}[Schneider-Stuhler]\label{thm-SS}
	Let  $\ell\neq p$ be a prime.  There is a $\GL_n(K)\times \cG_K$-equivariant isomorphism
	\[
		H^{r}_{\proet}(\HH^{n-1}_{C}, \Ql(r))\cong \Sp_r(\Ql)^{\ast}
	\] 
	for every integer $r\geq 0$. 
\end{thm}

As for the $p$-adic case, we have the following result of Colmez--Dospinescu--Nizio{\l}. 
\begin{thm}[{\cite[Theorem~5.13]{CDN2020-1}}]\label{thm-CDN}
	There is a strictly exact sequence of $\GL_n(K)\times \GG_K$-Fr\'echet spaces 
	\[
		0\rightarrow \Omega ^{r-1}(\HH^{n-1}_{C})/\ker d\rightarrow H^{r}_{\proet}(\HH^{n-1}_{C},\Qp(r))\rightarrow \Sp_r(\Qp)^{\ast}\rightarrow 0
	\] 
	for all integers $r\geq 0$. 
\end{thm}

\subsection{Topology and condensed mathematics}\label{subsection-3.2}
After seeing the various cohomology groups from the previous section, one might  ask: what is the topology on such cohomology groups, or more fundamentally, why should these cohomology groups possess a natural topology at all? The theory of condensed mathematics developed by Clausen and Scholze provides a natural way to endow cohomology groups with a topology. In this section, we briefly review some  notions in condensed mathematics (following \cite{CS2019-1}) that we will need later  and outline how a natural topology arises for cohomology groups under consideration.

Let $\ast_{\proet}$ be the pro-\'etale site of a point, which is equivalent to the category of profinite sets. A condensed set/group/ring is a sheaf over $\ast_{\proet}$ with values in $\Set/\Grp/\Ring$. We denote the category of condensed sets by $\Cond$.  There is a functor from the category $\Top$ of topological spaces  to the category  $\Cond$ given by 
\begin{align*}
	(\_): \Top &\rightarrow \Cond\\
	T&\mapsto (\underline{T}: S\mapsto \Cont(S,T))
\end{align*}
where $S$ is any profinite set and $\Cont(S,T)$ is the set of continuous maps from $S$ to $T$. We let  $\Cond(\Ab)$ be the category of condensed abelian groups, which is an abelian category \cite[Theorem~1.10]{CS2019-1} and has an internal Hom-functor denoted by $\ul{\Hom}$.  We also let $\Solid\subseteq \Cond(\Ab)$ be the subcategory of solid abelian groups (see \cite[Definition~5.1]{CS2019-1}), and tensor product of objects in $\Solid$ will be denoted by $\otimes^{\square}$.

Inside the category of profinite sets, we have the subcategory of extremally disconnected sets. These are the projective objects in the category $\ast_{\proet}$, and we denote this subcategory by $\EDis$. Given a condensed ring $R$, we denote by $\Mod^{\cond}_{R}$ the category of condensed $R$-modules. A \textit{pre-analytic ring} $\cA$ consists of a condensed ring $\ul{\cA}$ and a functor 
\begin{align*}
	\EDis&\rightarrow \Mod^{\cond}_{\ul{\cA}}\\
	S&\mapsto \cA[S]
\end{align*}
sending finite disjoint unions to finite products, together with a natural transformation $S\rightarrow \cA[S]$. Now given a pre-analytic ring $\cA$, if for any complex 
 $
C: \cdots \rightarrow C_i\rightarrow C_1\rightarrow C_0\rightarrow 0
$ 
in $\Ch(\Mod^{\cond}_{\ul{\cA}})$ with each  $C_i$ a direct sum of objects of the form $\cA[T]$ for varying $T\in \EDis$, the map 
$
	R\ul{\Hom}_{\ul{\cA}}(\cA[S], C)\rightarrow R\ul{\Hom}_{\ul{\cA}}(\ul{\cA}[S], C)
$ 
is an isomorphism in the derived category $D(\Cond(\Ab))$, then we say $\cA$ is \textit{analytic}. 

\begin{exmp}
	The primal example of a pre-analytic ring is  $\Z_{\square}$, whose underlying condensed ring is given by $\ul{\Z_{\square}}=\Z$ and the associated functor is defined as
	\begin{align*}
		\EDis &\rightarrow \Mod^{\cond}_{\Z}\\
		S=\invlim_{i}S_i&\mapsto \Z_{\square}[S]:=\invlim_{i} \Z[S_i].
	\end{align*}
	In fact, $\Z_{\square}$ is also  analytic and we have $\Solid=\Mod^{\cond}_{\Z_{\square}}$ (see \cite[Theorem~5.8]{CS2019-1}).
\end{exmp}

\begin{exmp}
	For $K/\Qp$ a finite extension, we can give $\OO_K$ an analytic ring structure $\OO_{K,\square}$ by setting  $\ul{\OO_{K,\square}}=\OO_K$ and defining $\OO_{K,\square}[S]=\invlim_{i}\OO_K[S_i]$ for $S=\invlim_{i}S_i\in \EDis$. We  can also give $K$ an anlytic ring structure $K_{\square}$ by setting $\ul{K_{\square}}=K$ and $K_{\square}[S]:=\OO_{K,\square}[S][\frac{1}{p}]$.
\end{exmp}

For the category $\Mod^{\cond}_{\Z_{\square}}$, we may also write it as $\Mod^{\solid}_{\Z}$ to emphasize solidity. When $K$ is a finite extension of $\Qp$, we let $\Mod^{\solid}_{K}$ denote the category of solid $K$-modules. Objects in $\Mod^{\solid}_{K}$ are solid abelian groups which are also $K$-vector spaces. 

We recall the following  classes of $K$-vector spaces from classical functional analysis.

\begin{defns}
	Let $V$ be a  topological $K$-vector space. 
\begin{arenumerate}
	\item We say $V$ is a \textit{$K$-Banach space} if $V$ admits a $\pi$-adically complete $\OO_{K}$-module $V^{0}\subseteq V$ such that $V^{0}\otimes_{\OO_{K}}K = V$ and $V^{0}/\pi^{n}V^{0}$ is discrete for all $n\in \N$. 
	\item We say $V$ is a \textit{$K$-Fr\'echet space} if $V$ is locally convex, complete, and its topology can be given via a countable family of seminorms. 
\end{arenumerate}
\end{defns}

We call an object in $\Mod^{\solid}_{K}$  a  \textit{solid $K$-Banach/Fr\'echet space}  if it arises as $\underline{V}$ for some $K$-Banach/Fr\'echet space $V$. 

Now for the cohomology groups appeared in Section \ref{subsection-3.1}, one way to endow them with a topology is by upgrading the constant sheaf $\Qp$  to a sheaf of condensed abelian groups  $\ul{\Qp}$ (with the $p$-adic topology). For  an adic space  $X$, this  will also upgrade the  pro-\'etale cohomology group $H^{i}_{\proet}(X_C, \Qp)$ into a condensed abelian group $H^{i}_{\cond}(X_{C,\proet}, \ul{\Qp})$. If there exists a topological $\Qp$-vector space $V$ such that 
 \[
	 H^{i}_{\cond}(X_{C,\proet}, \ul{\Qp})\cong \ul{V}
\] 
in $\Cond(\Ab)$, the topology of $V$ then  offers us a way to topologize $H^{i}_{\proet}(X_{C},\Qp)$. 

When Colmez--Dospinescu--Nizio{\l} first studied the $p$-adic (pro-)\'etale cohomology of Drinfeld spaces, the theory of condensed mathematics was not fully available yet. A condensed version of their calculations is contained in  \cite{Bosco2023-1}, which describes the cohomology groups and their topology  in a unified manner. 

\begin{exmp}
	We give an example of how one can endow a pro-\'etale cohomology group a natural topology through the case of  $H^{r}_{\proet}(\cH^{n-1}_{C},\Ql(r))$ in Theorem \ref{thm-SS}. Following the condensed approach in \cite{Bosco2023-1}, once we upgrade the sheaf $\Ql(r)$ into a sheaf of condensed abelian groups $\ul{\Ql(r)}$, the isomorphism in Theorem \ref{thm-SS} becomes 
	\[
		H^{r}_{\cond}(\cH^{n-1}_{C,\proet}, \ul{\Ql(r)})\cong \ul{\Hom}_{\Ql}(\ul{\Sp_r(\Ql)},\ul{\Ql}),
	\] 
	where $\Ql$ is equipped with the $\ell$-adic topology and the topology of $\Sp_r(\Ql)$ is induced by that of $\Ql$. One can then show that 
	\[
		\ul{\Hom}_{\Ql}(\ul{\Sp_r(\Ql)}, \ul{\Ql})\cong \ul{\Hom_{\cts}(\Sp_r(\Ql),\Ql)}
	\] 
	in the category $\Mod^{\solid}_{\Ql}$, where $\Hom_{\cts}(\Sp_r(\Ql),\Ql)$ is equipped with the weak topology. The topology of $H^{r}_{\proet}(\cH^{n-1}_{C},\Ql(r))$ is then given by that of $\Hom_{\cts}(\Sp_r(\Ql),\Ql)$, which is a $\Ql$-Fr\'echet space. 
\end{exmp}

Using the condensed formalism, we now elaborate on the Hochschild-Serre spectral sequence for pro-\'etale cohomology. Let $G$ be a profinite group, $X$ a rigid analytic space over $K$, and $Y\rightarrow X$ a pro-\'etale $G$-torsor. The pro-\'etale cohomology of $X$ can be computed via the simplicial cover 
\[
	\cdots\mathrel{\substack{\textstyle\rightarrow\\[-0.3ex]
                      \textstyle\rightarrow \\[-0.3ex]
                      \textstyle\rightarrow}} Y \times_X Y \rightrightarrows Y
\] 
where the $(n+1)$-fold product $(Y/X)^{n+1}$ is isomorphic to $Y\times G^{n}$. By \cite[Lemma~4.4]{CGN} (see also \cite[Proposition~4.12]{Bosco2023-1}), we have a quasi-isomorphism in $D(\Solid)$
\[
	R\Gamma_{\cond}((Y\times G^{n})_{\proet}, \underline{\Qp})\simeq R\underline{\Hom}(\Z_{\square}[G^{n}], R\Gamma_{\cond}(Y_{\proet}, \underline{\Qp})). 
\] 
This gives us a Hochschild-Serre spectral sequence 
\[
	E^{i,j}_{2}=H^{i}_{\cond}(G, H^{j}_{\cond}(Y_{\proet}, \underline{\Qp}))\Rightarrow H^{i+j}_{\cond}(X_{\proet}, \underline{\Qp})
\] 
or equivalently, if we remember that the cohomology groups have a topology,
\[
	E^{i,j}_{2}=H^{i}_{\cts}(G, H^{j}_{\proet}(Y, \Qp))\Rightarrow H^{i+j}_{\proet}(X, \Qp).
\] 
Such a spectral sequence allows us to compute the cohomology of a rigid analytic variety $X$ over $K$ using the $\cG_K$-torsor $X_{C}\rightarrow X$ and also the cohomology of a quotient stack $[X/G]$ using the $G$-torsor $X\rightarrow [X/G]$.

In the rest of the paper, we will not emphasize the use of condensed formalism to prevent overloading the notation with too many "underlines". Instead, we will simply state what is  the topology of the cohomology groups under consideration, and the reader should keep in mind that the  topology comes from condensing the relevant objects.

\section{Cohomology of the isomorphism stacks}\label{section-4}
In this section, we compute the pro-\'etale cohomology of the isomorphism stack $[\cH^{n-1}_{K}/\GL_n(\OO_K)]$. 
\subsection{\texorpdfstring{$\ell$-adic case}{l-adic case}}\label{subsection-4.1}
First recall the  Galois cohomology of $\cG_K$ over $\Ql$, which can be summarized as follows. 
\begin{lem}\label{lem-l-Gal}
	For $j\in \Z$, let $\Ql(j)$ be the $j$-th Tate twist of $\Ql$. We have  
	\begin{align*}
		&H^{0}_{\cts}(\GG_K, \Ql(j))\cong\begin{cases}
	\Ql  &  \text{if $j=0$}\\
	0  &  \text{otherwise}
	\end{cases} \\
		&H^{1}_{\cts}(\GG_K, \Ql(j))\cong\begin{cases}
		\Ql  &  \text{if $j=0,1$}\\
		0  &  \text{otherwise}
		\end{cases}\\
		&H^{2}_{\cts}(\GG_K, \Ql(j))\cong\begin{cases}
		\Ql  &  \text{if $j=1$}\\
		0  &  \text{otherwise}
		\end{cases}
	\end{align*}
and  higher cohomology groups vanish for all $j$. 
\end{lem}
\begin{proof}
	Since $\GG_K$ acts on $\Ql(j)$ through the $j$-th power $\ell$-adic cyclotomic character, we see that $H^{0}$ is nonzero only when $j=0$. By local Tate duality,  $H^{i}_{\cts}(\GG_K, \Ql(j))$ and $H^{2-i}_{\cts}(\GG_K, \Ql(1-j))$ are dual to each other,which  gives the description of $H^{2}$. Using the local Euler-Poincar\'e characteristic formula \cite[Corollary~7.3.8]{coh_number_fields}, we have
	$\sum^{2}_{i=0}(-1)^{i}\dim_{\Ql}H^{i}_{\cts}(\GG_K, \Ql(j))=0$ 
	for any fixed $j$, and this gives the description of $H^{1}$. Finally, $H^{i}_{\cts}(\GG_K, \Ql(j))$ vanishes for all $i\geq 3$ and all $j$ as the cohomological degree of $\GG_K$ is 2. 
\end{proof}

Now we compute the $\ell$-adic pro-\'etale cohomology of Drinfeld spaces over $K$ through Galois descent. 
\begin{prop}\label{prop-l-descent}
	There are $\GL_n(K)$-equivariant isomorphisms 
\[
	H^{r}_{\proet}(\HH ^{n-1}_K, \Ql)\cong \begin{cases}
		\Ql &  \text{if $r=0,1$}\\
	0  &  \text{otherwise.}
	\end{cases}
\] 
\end{prop}
\begin{proof}
	We use the Hochschild-Serre spectral sequence to carry out Galois descent from the geometric cohomology $H^{r}_{\proet}(\HH^{n-1}_{C},\Ql)$ to the arithmetic cohomology $H^{r}_{\proet}(\HH^{n-1}_{K},\Ql)$, where the  $E_2$-page is given by 
	\[
		E_{2}^{i,j}=H^{i}_{\cts}(\GG_K, H^{j}_{\proet}(\HH^{n-1}_{C},\Ql))\Rightarrow H^{i+j}_{\proet}(\HH^{n-1}_{K},\Ql).
	\] 
	By Theorem \ref{thm-SS}, $H^{j}_{\proet}(\HH^{n-1}_{C},\Ql)\cong \Sp_j(\Ql)^{\ast}(-j)$ where $\GG_K$ acts through the Tate twist. So we have  $E^{i,j}_{2}=H^{i}_{\cts}(\GG_K, \Ql(-j))\otimes \Sp_j(\Ql)^{\ast}$, which is isomorphic to $\Sp_j(\Ql)^{\ast}$ if $j=0$ and $i=0,1$ and is 0 for all other $i,j$ by Lemma \ref{lem-l-Gal}. Thus the spectral sequence  degenerates on the $E_2$-page and the claim follows.
\end{proof}

To compute the $\ell$-adic pro-\'etale cohomology of the stack $[\cH^{n-1}_{K}/\GL_n(\OO_K)]$, we first compute the continuous group cohomology of $\GL_n(\OO_K)$ over $\Ql$. 
\begin{lem}\label{lem-int-l}
	Let $\Ql$ be the trivial 1-dimensional $\GL_n(\OO_K)$-representation equipped with the $\ell$-adic topology. Then 
	\[
		H^{i}_{\cts}(\GL_n(\OO_K), \Ql)=\begin{cases}
		\Ql  &  \text{if $i=0$}\\
		0  &  \text{otherwise.}
		\end{cases}
	\] 
\end{lem}
\begin{proof}
	Let  $P\subseteq \GL_n(\OO_K)$ be a pro-$p$ subgroup of finite index $m$. Since $\ell\neq p$, the $\ell$-cohomological dimension of $P$ is 0 \cite[Corollary~3.3.7]{coh_number_fields}. Thus we have  $H^{i}_{\cts}(P, \Z/\ell^{r}\Z)=0$ for $i,r>0$. This implies the system $\{H^{i}_{\cts}(P, \Z/\ell^{r}\Z)\}_{r}$ is Mittag-Leffler for all $i\geq 0$. Thus we get $R^{1}\invlim_{r}H^{i-1}_{\cts}(P, \Z/\ell^{r}\Z)=0$ and $H^{i}_{\cts}(P, \Zl)\cong \invlim_{r}H^{i}_{\cts}(P, \Z/\ell^{r}\Z)=0$ for $i>0$. Using the $\ell$-cohomological dimension of $P$ again, we also have $H^{i}_{\cts}(P, \Ql/\Zl)=0$ for $i>0$. Together, such vanishing results imply $H^{i}_{\cts}(P, \Ql)=0$ for $i>0$. 

	Now let $\cores^{\GL_n(\OO_K)}_{P}$ and $\res^{\GL_n(\OO_K)}_{P}$ denote the corestriction and restriction maps on group cohomology and also let $[m]$ denote the multiplication-by-$m$ map. Using the vanishing of $H^{i}_{\cts}(P, \Ql)$, we get 
	\[
		[m]|_{H^{i}_{\cts}(\GL_n(\OO_K), \Ql)}=\cores^{\GL_n(\OO_K)}_{P}\circ \res^{\GL_n(\OO_K)}_{P}=0\quad \text{for $i>0$.}
	\] 
	As $m$ is invertible in $\Ql$, we get $H^{i}_{\cts}(\GL_n(\OO_K), \Ql)=0$ for $i>0$.
\end{proof}

\begin{thm}\label{thm-l-iso}
There are isomorphisms
\[
	H^{r}_{\proet}([\HH^{n-1}_K/\GL_n(\OO_K)],\Ql)\cong \begin{cases}
	\Ql  &  \text{if $r=0,1$}\\
	0  &  \text{otherwise.}
	\end{cases}
\] 
\end{thm}
\begin{proof}
We compute the pro-\'etale cohomology of the stack using the spectral sequence 
\[
	E^{i,j}_{2}=H^{i}_{\cts}(\GL_n(\OO_K), H^{j}_{\proet}(\HH^{n-1}_{K}, \Ql))\Rightarrow H^{i+j}_{\proet}([\HH^{n-1}_{K}/\GL_n(\OO_K)], \Ql).
\] 
	By Proposition \ref{prop-l-descent} and Lemma \ref{lem-int-l}, we get $E^{i,j}_{2}\cong \Ql$ when $j=0,1$ and $i=0$ and vanishes everywhere else. This finishes the computation.  
\end{proof}

\subsection{\texorpdfstring{$p$-adic case}{p-adic case}}\label{subsection-4.2}
As one might have anticipated, the steps used in the $\ell$-adic case do not  translate to the $p$-adic case. Specifically, as we will see later in Proposition \ref{prop-p-descent}, while the differential part appeared  in Theorem \ref{thm-CDN} does not survive the Galois descent from $C$ to $K$, the Galois cohomology of $\cG_K$ over $\Qp(j)$ is  more intricate than over $\Ql(j)$, making duals of nontrivial Steinberg representations appear in cohomology of Drinfeld spaces after descent. The challenge then comes from computing the group cohomology of $\GL_n(\OO_K)$ over representations $\Sp_r(\Qp)^{\ast}$.

To proceed with our computation of  $p$-adic pro-\'etale cohomology of $[\cH^{n-1}_{K}/\GL_n(\OO_K)]$, we will pass to the Lubin-Tate side and compute the $p$-adic pro-\'etale cohomology of the stack $[\LT_{n,\breve{K}}/\G_n]$ instead. As implied by Corollary \ref{cor-iso}, cohomology groups of these two quotient stacks are identical.

We again start by recalling  Galois cohomology of $\cG_K$ over Tate twists of $\Qp$.
\begin{lem}\label{lem-p-Gal}
Let $d$ be the degree of $K$ over $\Qp$. Then for $j\in \Z$, we have  
	\begin{align*}
		&H^{0}_{\cts}(\GG_K, \Qp(j))\cong\begin{cases}
	\Qp  &  \text{if $j=0$}\\
	0  &  \text{otherwise}
	\end{cases}\\
	&H^{1}_{\cts}(\GG_K, \Qp(j))\cong\begin{cases}
	\Qp^{d+1}  &  \text{if $j=0,1$}\\
	\Qp^{d}  &  \text{otherwise}
	\end{cases} \\
	&H^{2}_{\cts}(\GG_K, \Qp(j))\cong\begin{cases}
	\Qp  &  \text{if $j=1$}\\
	0  &  \text{otherwise}
	\end{cases}
	\end{align*}
	and  higher cohomology groups vanish for all $j$. 
\end{lem}
\begin{proof}
	Same as in Lemma \ref{lem-l-Gal}, $H^{0}$ can be computed directly and local Tate duality gives the description of $H^{2}$. For all fixed $j$, the  local Euler-Poincar\'e  characteristic formula \cite[Corollary~7.3.8]{coh_number_fields} now implies $\sum^{2}_{i=0}(-1)^{i}\dim_{\Qp} H^{i}_{\cts}(\GG_K, \Qp(j))=-[K:\Qp]$, which gives the description of $H^{1}$. 
\end{proof}

As for Galois cohomology over Tate twists of $C$, we have the following theorem of Tate. 
\begin{thm}[\cite{Tate1967-1}]\label{thm-Tate}
	Let $K$ be a finite extension of $\Qp$ and let $C(j)$ be the 1-dimensional $\GG_K$-representation over $C$ where $\GG_K$ acts through the $j$-th power cyclotomic character. Then 
	\[
		H^{i}_{\cts}(\GG_K, C(j))\cong \begin{cases}
		K  &  \text{if $j=0$ and $i=0,1$}\\
		0  &  \text{otherwise.}
		\end{cases}
	\] 
\end{thm}

On the Lubin-Tate side, since we are quotienting  $\LT_{n,\breve{K}}$ by the Morava stabilizer group $\G_n\cong \OO_D^{\times}\rtimes\hat{\Z}$, we evidently need to compute the continuous group cohomology of $\OO^{\times}_{D}$ over $\Qp$. Let $D_0$ be the division algebra of invariant $1/n$ over $\Qp$. 
The following result follows  from Lazard's comparison theorem \cite[Th\'eor\`eme~V.2.4.10]{Lazard1965-1} between continuous cohomology and Lie algebra cohomology for $\Qp$-analytic groups.

\begin{lem}[{\cite[Proposition~3.8.1]{BSSW-1}}]\label{lem-Lazard}
	Let $G$ be either $\GL_n(\Zp)$ or $\OO_{D_0}^{\times}$ with the trivial action on $\Qp$. Then
	\[
		H^{\ast}_{\cts}(G, \Qp)\cong  H^{\ast}_{\Lie}(\gl_n \Qp, \Qp)\cong \Lambda _{\Qp}(x_1,x_3,\cdots,x_{2n-1})
	\] 
	as graded $\Qp$-algebras where $\left|x_i\right|=i$ in the exterior algebra on the right. 
\end{lem}

In Lemma \ref{lem-normal-closure} and Proposition \ref{prop-generalization} below, we generalize Lazard's result from $\Qp$-analytic groups to $K$-analytic groups for finite extensions $K/\Qp$. Our strategy is motivated by the proof of \cite[Lemma~4.1.1]{HKN}. 
\begin{lem}\label{lem-normal-closure}
	Let $K$ be a finite extension of $\Qp$ of degree $d$ and let $\GL_n$ be the general linear algebraic group defined over $K$, then
	\[
		(\Res_{K/\Qp}\GL_n)_{\ol{K}}\cong (\GL_{n,\ol{K}})^{d}
	\] 
	as group schemes over $\ol{K}$.
\end{lem}
\begin{proof}
	Since $[K:\Qp]=d$, there are exactly $d$ different embeddings of $K$ into $\ol{K}$ over  $\Qp$. Denote by $\Sigma =\{\sigma_i \}$ the set of embeddings of $K$ into $\ol{K}$. Then we have $\ol{K}\otimes_{\Qp} K\cong \prod^{}_{\sigma _i\in \Sigma } \ol{K}$. 

	 Now we proceed by comparing the functor of points associated with the two group schemes in question. Let $X$ be a scheme over $\ol{K}$, then for each embedding $\sigma _i: K\hookrightarrow \ol{K}$, we have a corresponding map $f_i: \ol{K}\rightarrow \ol{K}\otimes_{\Qp}K\cong \prod^{}_{\sigma _i\in \Sigma }\ol{K}$ onto the copy of $\ol{K}$ associated with $\sigma _i$. This induces a map 
	 \[
		f_i: \Hom_{\ol{K}}(X, \GL_{n,\ol{K}})\rightarrow \Hom_{\ol{K}} (X\times_{\ol{K}}\Spec (\ol{K}\otimes_{\Qp}K), \GL_{n,\ol{K}})
	 \] 
	 where 
	 \begin{align*}
		 \Hom_{\ol{K}}(X\times_{\ol{K}}\Spec (\ol{K}\otimes_{\Qp}K), \GL_{n,\ol{K}})&=\Hom_{K}(X\times \Spec K, \GL_{n})\\
		 &= \Hom_{\Qp}(X, \Res_{K/\Qp}\GL_{n})\\
		 &= \Hom_{\ol{K}}(X, (\Res_{K/\Qp} \GL_n)_{\ol{K}}).
	 \end{align*}
	 At the same time, the projection $s_i: \prod^{}_{\sigma _i\in \Sigma }\ol{K}\rightarrow \ol{K}$ mapping the copy of $\ol{K}$ corresponding to $\sigma _i$ identically to $\ol{K}$ gives a section of $f_i$. Thus we also get a map 
	  \[
		  s_i: \Hom_{\ol{K}}(X, (\Res_{K/\Qp}\GL_n)_{\ol{K}})\rightarrow \Hom_{\ol{K}}(X, \GL_{n,\ol{K}})
	 \] 
	 such that $s_i\circ f_i=\id$. This implies that $\GL_{n,\ol{K}}$ is a direct summand of $(\Res_{K/\Qp}\GL_n)_{\ol{K}}$. As this is true for all $\sigma _i$, we conclude that $(\Res_{K/\Qp}\GL_n)_{\ol{K}}\cong (\GL_{n,\ol{K}})^{ d}$. 
\end{proof}

\begin{prop}\label{prop-generalization}
	Suppose  $K/\Qp$ is  a finite extension of degree $d$ and $D$ is a division algebra over $K$ of invariant $1/n$. Let $G$ be either $\GL_n(\OO_K)$ or $\OO_D^{\times}$ with a trivial action on $\Qp$. Then we have
	\[
		H^{\ast}_{\cts}(G,\Qp)\cong \Lambda _{\Qp}(x_1,\cdots,x_{2n-1})^{\otimes d}
	\] 
	as graded $\Qp$-algebras where $\left|x_i\right|=i$ in the exterior algebra on the right.
\end{prop}
\begin{proof}
	We first consider the case of  $G=\GL_n(\OO_K)$, which can be viewed as a $\Qp$-analytic group via restriction of scalars. The Lie algebra of $G$ is then given by $\Res_{K/\Qp}\gl_nK$, which we denote by $\g$. By Lemma \ref{lem-normal-closure}, we have
	$
		\g\otimes_{\Qp} \Qpbar\cong (\gl_n \Qpbar)^{d}
	$ 
	as Lie algebras over $\Qpbar$. This gives
	\begin{equation}\label{equation-closure}
		H^{\ast}_{\Lie} (\g\otimes \Qpbar, \Qpbar)\cong H^{\ast}_{\Lie}((\gl_n \Qpbar)^{d}, \Qpbar)\cong \Lambda _{\Qpbar}(x_1,x_3,\cdots,x_{2n-1})^{\otimes d}
	\end{equation}
	where the second isomorphism follows from Lemma \ref{lem-Lazard} and the product formula for Lie algebra cohomology. In the meantime, base-change gives us an isomorphism 
	\begin{equation}\label{eq-base_change}
		H^{\ast}_{\Lie}(\g, \Qp)\otimes \Qpbar\cong H^{\ast}_{\Lie}(\g\otimes \Qpbar, \Qpbar).
	\end{equation}
	Thus we have $H^{\ast}_{\Lie}(\g, \Qp)\otimes \Qpbar\cong \Lambda_{\Qpbar}(x_1,x_3,\cdots, x_{2n-1})^{\otimes d}$ and it suffices to show the latter description descends to $\Qp$. 

	We claim that the base-change isomorphism (\ref{eq-base_change}) preserves primitive elements. Recall from \cite[\S~10]{Kos50} that given a Lie algebra $\fa$ over a field $F$, its Lie algebra cohomology carries a coproduct $\Delta : H^{\ast}_{\Lie}(\fa, F)\rightarrow H^{\ast}_{\Lie}(\fa, F)\otimes H^{\ast}_{\Lie}(\fa, F)$ and the primitive elements in $H^{\ast}_{\Lie}(\fa, F)$ are those such that $\Delta(x)=x\otimes 1+1\otimes x$. The map $\Delta $ comes  from the diagonal map $D: \fa\rightarrow \fa\times \fa$ of Lie algebras. As the diagonal map is compatible with base-change, we see the coproduct $\Delta$ commutes with base-change, hence the claim.

	In the graded algebra $\Lambda_{\Qpbar}(x_1,x_3,\cdots,x_{2n-1})^{\otimes d}$, the  group of automorphisms which preserve primitive elements is given by $\prod^{}_{i\in \{1,3,\cdots, 2n-1\}}\GL_d(\Qpbar)$.
	Using the claim above, to obtain  the descent from $\Qpbar$ to $\Qp$,  it is enough to show $H^{1}(\cG_{\Qp}, \GL_d(\Qpbar))$ is trivial. As $\cG_{\Qp}$ is profinite, any continuous 1-cocycle $\gamma: \cG_{\Qp}\rightarrow \GL_d(\Qpbar)$ has finite image. Thus the image of $\gamma $ is contained in $\GL_d(L)$ for some finite Galois extension $L/\Qp$, and $\gamma$ factors through the finite Galois group $\Gal(L/\Qp)$. Then using non-abelian Hilbert 90 (\cite[\S III.1~Lemma 1]{Serre-Galois}), we get 
	\begin{equation}\label{equation-90}
		H^{1}(\cG_{\Qp}, \GL_d(\Qpbar))\cong \dirlim_{L/\Qp \text{ finite Galois}}H^{1}(\Gal(L/\Qp), \GL_d(L))=0.
	\end{equation}
	This implies $H^{\ast}_{\Lie}(\fg, \Qp)\cong \Lambda_{\Qp}(x_1,x_3,\cdots, x_{2n-1})^{\otimes d}$, which proves the statement for $\GL_n(\OO_K)$.

	When $G=\OO_D^{\times}$, the Lie algebra of $G$ and the Lie algebra of  $\GL_n(\OO_K)$ are isomorphic  over $\Qpbar$. Thus by (\ref{equation-closure}) and (\ref{equation-90}),  the continuous group cohomology of $\OO_D^{\times}$ over $\Qp$ is the same as that of $\GL_n(\OO_K)$.  
\end{proof}

We will need the following projection formula in our later computation. 
\begin{lem}\label{lem-proj}
	Let $G$ be a profinite group and $\cM,\cN\in \Mod^{\solid}_{K[G]}$. If the action of $G$ on $\cN$ is trivial and $\cN$ is a solid  $K$-Fr\'echet space, there is an isomorphism 
	\[
		H^{\ast}(G, \cM\otimes^{\square}_K \cN)\cong H^{\ast}(G,\cM)\otimes^{\square}_{K}\cN.
	\] 
Specifically, if $M,N$ are  $K$-Fr\'echet  spaces and the action of $G$ on $N$ is trivial, then we have an isomorphism
\[
	H^{\ast}_{\cts}(G, M\wh{\otimes}_K N)\cong H^{\ast}_{\cts}(G, M)\wh{\otimes}_K N.
\] 
\end{lem}
\begin{proof}
	Consider $\cM,\cN$ as objects in $D(\Mod^{\solid}_{K[G]})$ and let $R\Gamma(G,-)$ be the derived functor of $(-)^{G}: \Mod^{\solid}_{K[G]}\rightarrow \Mod^{\solid}_{K}$. Since the action of $G$ on $\cN$ is trivial, we have an isomorphism
	\begin{equation}\label{equation-tensor}
		R\Gamma(G, \cM\otimes^{L\square}_{K}\cN)\cong R\Gamma(G,\cM)\otimes^{L\square}_{K}\cN
	\end{equation}
	in the category $D(\Mod^{\solid}_{K})$ where $\otimes^{L\square}_{K}$ is the derived tensor product in $D(\Mod^{\solid}_{K})$.
	Because $\cN$ is a $K$-Fr\'echet space, it is flat for the tensor product  $\otimes^{\square}_{K}$ by \cite[Corollary~A.65]{Bosco2023-1}. Therefore, we have $\cM\otimes^{L\square}_{K}\cN\cong \cM\otimes^{\square}_{K}\cN$ in $D(\Mod^{\solid}_{K[G]})$ and the cohomology of the left-hand side in (\ref{equation-tensor}) calculates $H^{\ast}(G, \cM\otimes^{\square}_{K}\cN)$. The cohomology of the right-hand side in (\ref{equation-tensor}) admits a spectral sequence
	\[
		E^{p,q}_{2}=H^{p}(H^{q}(R\Gamma(G,\cM))\otimes^{L\square}_{K}\cN)\Rightarrow H^{p+q}(R\Gamma(G,\cM)\otimes^{L\square}_{K}\cN).
	\] 
	Again as $\cN$ is flat, we have $H^{q}(R\Gamma(G,\cM))\otimes^{L\square}_{K}\cN\cong H^{q}(R\Gamma(G,\cM)) \otimes^{\square}_{K}\cN$. Thus the cohomology of the right-hand side in (\ref{equation-tensor}) calculates $H^{\ast}(G,\cM)\otimes^{\square}_{K}\cN$ and the first claimed isomorphism of cohomology rings follows.

	The second assertion follows directly from the first one together with the isomorphism between $\underline{V\wh{\otimes}_K W}$ and $\underline{V}\otimes^{\square}_{K}\underline{W}$ in the category $\Mod^{\solid}_{K}$ for $K$-Fr\'echet spaces $V,W$, cf. \cite[Proposition~A.68]{Bosco2023-1}.
\end{proof}

While our computation will be through the Lubin-Tate side, we nevertheless give a description of  the $p$-adic pro-\'etale cohomology of Drinfeld spaces over $K$ as an interlude. 
\begin{prop}\label{prop-p-descent}
	Suppose $K$ is a finite extension of $\Qp$ of degree $d$. There are $\GL_n(K)$-equivariant isomorphisms 
	\[
		H^{r}_{\proet}(\cH^{n-1}_{K},\Qp)\cong \begin{cases}
		\Qp  &  \text{if $r=0$}\\
		\Qp^{d+1}  &  \text{if $r=1$}\\
			\Sp_{r-1}(\Qp)^{\ast}\otimes \Qp^{d} &\text{if $r\geq 2$.}
		\end{cases}
	\] 
\end{prop}
\begin{proof}
	By Theorem \ref{thm-CDN}, we have a short exact sequence 
	\begin{equation}\label{equation-SES1}
		0\rightarrow \Omega ^{j-1}(\cH^{n-1}_{K})/\ker d \, \wh{\otimes}_K C(-j)\rightarrow H^{j}_{\proet}(\cH^{n-1}_{C},\Qp)\rightarrow \Sp_j(\Qp)^{\ast}(-j)\rightarrow 0
	\end{equation}
	for $j\geq 0$. Notice that when  $j=0$, the differential part vanishes and $H^{0}_{\proet}(\cH^{n-1}_{C},\Qp)$ is simply $\Qp$. Using Lemma \ref{lem-p-Gal}, we have
	\[
		H^{i}_{\cts}(\cG_K, H^{0}_{\proet}(\cH^{n-1}_{C},\Qp))\cong \begin{cases}
		\Qp  &  \text{if $i=0$}\\
		\Qp^{d+1}  &  \text{if $i=1$}\\
			0 & \text{if $i\geq 2$.}
		\end{cases}
	\]	
	When $j\geq 1$, the Galois cohomology of (\ref{equation-SES1}) gives a long exact sequence
	\begin{multline*}
		\cdots \rightarrow H^{i-1}_{\cts}(\cG_K, \Sp_j(\Qp)^{\ast}(-j))\rightarrow H^{i}_{\cts}(\cG_K, \Omega ^{j-1}(\cH^{n-1}_{K})/\ker d \, \wh{\otimes}_K C(-j))\\
		\rightarrow H^{i}_{\cts}(\cG_K, H^{j}_{\proet}(\cH^{n-1}_{K},\Qp))\rightarrow H^{i}_{\cts}(\cG_K, \Sp_j(\Qp)^{\ast}(-j))\rightarrow \cdots.
	\end{multline*} 
	Since the action of $\cG_K$ on $\Omega ^{j-1}(\cH^{n-1}_{K})/\ker d\, \wh{\otimes}_K C(-j)$ is through the Tate twist, we have 
	\[
		H^{i}_{\cts}(\cG_K, \Omega ^{j-1}(\cH^{n-1}_{K})/\ker d \, \wh{\otimes}_K C(-j))\cong H^{i}_{\cts}(\cG_K, C(-j))\wh{\otimes}_K \Omega ^{j-1}(\cH^{n-1}_{K})/\ker d
	\] 
	by Lemma \ref{lem-proj}, and Theorem \ref{thm-Tate} further implies that all such cohomology groups vanish. Since the Galois action on $\Sp_j(\Qp)^{\ast}(-j)$ is also through the Tate twist, using Lemma \ref{lem-p-Gal},  we get 
	\[
		H^{i}_{\cts}(\cG_K, H^{j}_{\proet}(\cH^{n-1}_{C},\Qp))\cong \begin{cases}
			\Sp_j(\Qp)^{\ast}\otimes \Qp^{d}  &  \text{if $i=1$}\\
		0  &  \text{otherwise.}
		\end{cases}
	\] 
	By applying these computations of Galois cohomology groups to the spectral sequence for Galois descent, we obtain the claimed isomorphisms. 
\end{proof}

Now we come back to the Lubin-Tate side and finish our computation of $H^{\ast}_{\proet}([\LT_{n,\breve{K}}/\G_n],\Qp)$.
\begin{thm}\label{thm-p-iso}
Let $K$ be a finite extension of $\Qp$ of degree $d$. There is an isomorphism of graded $\Qp$-algebras
\[
	H^{\ast}_{\proet}([\LT_{n,\breve{K}}/\G_n], \Qp)\cong \Lambda _{\Qp}(x_1,x_3,\cdots,x_{2n-1})^{\otimes d}\otimes H^{\ast}_{\cts}(\GG_K, \Qp)
\] 
where the degree of $x_i$ is $i$ in the exterior algebra on the right-hand side. 
\end{thm}
\begin{proof}
	The generic fiber $\LT_{n,\breve{K}}$ of the Lubin-Tate space is isomorphic to the rigid analytic $(n-1)$-dimensional open ball $\ocirc{\B}^{n-1}_{\breve{K}}$. However, doing Galois descent on the pro-\'etale cohomology of $\ocirc{\B}^{n-1}_{C}$ directly from $C$ to $\breve{K}$ can be difficult. Instead, we will stitch together the quotient by $\Gal(\ol{K}/\breve{K})$ and the quotient by the Morava stabilizer group $\G_n$. Over $\ocirc{\B}^{n-1}_{C}$, we have actions of both $\cG_K$ and $\G_n$. Since both groups surject onto $\Gal(\breve{K}/K)\cong \hat{\Z}$, we construct a group $H$ as their equalizer so $H$ fits into the diagram 
	\[
	\begin{tikzcd}
		H\arrow[r]\arrow[d] & \cG_K\arrow[d, "f_2"]\\
		\G_n\arrow[r, "f_1"] & \Gal(\breve{K}/K)
	\end{tikzcd}
	\] 
	where $\ker(f_1)=\OO_D^{\times}$, $\ker(f_2)=\Gal(\ol{K}/\breve{K})$, and we have a short exact sequence 
	\begin{equation}\label{equation-H}
	0\rightarrow \OO_D^{\times}\rightarrow H\rightarrow \cG_K\rightarrow 0.
	\end{equation}
	The construction of $H$ allows us to study the actions of $\cG_K$ and $\G_n$ together, and it identifies the pro-\'etale cohomology of $[\LT_{n,\breve{K}}/\G_n]$ with that of $[\ocirc{\B}^{n-1}_{C}/H]$. 

	We now compute the continuous group cohomology of $H$ over $H^{j}_{\proet}(\ocirc{\B}^{n-1}_{C},\Qp)$ for $j\geq 0$ using (\ref{equation-H}) and the associated Hochschild-Serre spectral sequence (see \cite[Section~3.3]{BSSW-1} for a justification of the Hochschild-Serre spectral sequence in the setting of continuous cohomology).  When $j=0$, we have $H^{0}_{\proet}(\ocirc{\B}^{n-1}_{C},\Qp)\cong \Qp$ by Theorem \ref{thm-ball}. Then Lemma \ref{lem-p-Gal} and Proposition \ref{prop-generalization} imply that only the 0th column and the 1st column in the spectral sequence $E^{p,q}_{2}=H^{p}_{\cts}(\cG_K, H^{q}_{\cts}(\OO^{\times}_{D},\Qp))\Rightarrow H^{p+q}_{\cts}(H, \Qp)$ are nonzero. Therefore, the spectral sequence degenerates and  gives
	\[
		H^{\ast}_{\cts}(H, H^{0}_{\proet}(\ocirc{\B}^{n-1}_{C}, \Qp))\cong \Lambda _{\Qp}(x_1,x_3,\cdots,x_{2n-1})^{\otimes d} \otimes H^{\ast}_{\cts}(\cG_K,\Qp). 
	\] 
	When $j\geq 1$, Theorem \ref{thm-ball} implies that $H^{j}_{\proet}(\ocirc{\B}^{n-1}_{C},\Qp)\cong \Omega^{j-1}(\ocirc{\B}^{n-1}_{K})/\ker d \, \wh{\otimes}_K C(-j)$. Thus in the spectral sequence $E^{p,q}_{2}=H^{p}_{\cts}(\cG_K, H^{q}_{\cts}(\OO^{\times}_{D}, H^{j}_{\proet}(\ocirc{\B}^{n-1}_{C},\Qp)))\Rightarrow H^{p+q}_{\cts}(H, H^{j}_{\proet}(\ocirc{\B}^{n-1}_{C},\Qp))$, the term $E^{p,q}_{2}$ is given by 
	\begin{align*}
		E^{p,q}_{2} &= H^{p}_{\cts}(\cG_K, H^{q}_{\cts}(\OO^{\times}_{D}, \Omega ^{j-1}(\ocirc{\B}^{n-1}_{K})/\ker d\, \wh{\otimes}_K C(-j)))\\
		&= H^{p}_{\cts}(\cG_K, H^{q}_{\cts}(\OO^{\times}_{D}, \Omega ^{j-1}(\ocirc{\B}^{n-1}_{K})/\ker d)\wh{\otimes}_K C(-j))\\
		&= H^{p}_{\cts}(\cG_K, C(-j))\wh{\otimes}_K H^{q}_{\cts}(\OO^{\times}_{D}, \Omega ^{j-1}(\ocirc{\B}^{n-1}_{K})/\ker d).
	\end{align*}
	Using the projection formula from Lemma \ref{lem-proj}, the second equality follows from the fact that $\OO^{\times}_D$ only acts on the differential forms, and the third equality follows from the fact that  $\cG_K$ acts through the Tate twist. By Theorem \ref{thm-Tate}, we get $E^{p,q}_{2}=0$ and $H^{\ast}_{\cts}(H, H^{j}_{\proet}(\ocirc{\B}^{n-1}_{C},\Qp))=0$ when $j\geq 1$. 

	Finally, inserting these computations into the spectral sequence 
	\[
		E^{i,j}_{2}=H^{i}_{\cts}(H, H^{j}_{\proet}(\ocirc{\B}^{n-1}_{C},\Qp))\Rightarrow H^{i+j}_{\proet}([\ocirc{\B}^{n-1}_{C}/H],\Qp),
	\] 
	we get the isomorphisms as desired. 
\end{proof}
Since $H^{\ast}_{\cts}(\cG_K,\Qp)$ is already calculated in Lemma \ref{lem-p-Gal}, Theorem \ref{thm-p-iso} gives an explicit description of $p$-adic pro-\'etale cohomology of the stack $[\LT_{n,\breve{K}}/\G_n]$ in all degrees. 

\begin{rem}
	Following the same steps as in the proof of Theorem \ref{thm-p-iso}, one can also compute the $\ell$-adic pro-\'etale cohomology of $[\cH^{n-1}_{K}/\GL_n(\OO_K)]$ by passing to the Lubin-Tate side, recovering the description given in Theorem \ref{thm-l-iso}. 
\end{rem}

\section{Cohomology of the isogeny stacks}\label{section-5}
We now turn to the computation of  pro-\'etale cohomology of the isogeny stack $[\cH^{n-1}_{K}/\GL_n(K)]$, where many of the auxiliary results from Section \ref{section-4} are still helpful.  

\subsection{\texorpdfstring{$\ell$-adic case}{l-adic case}}
We will compute $H^{\ast}_{\proet}([\cH^{n-1}_{K}/\GL_n(K)],\Ql)$ directly from the Drinfeld side, where the key step is to determine the continuous cohomology of $\GL_n(K)$ over $\Ql$. The cohomology of $\GL_n(K)$ over the trivial module over a ring $R$ has been well-studied when $R$ is equipped with the discrete topology. For example, \cite[Corollaire~2.1.7]{Dat} (see also \cite[Corollary~2]{Orlik2005-1}) gives 
\begin{equation}\label{eq-classical}
	H^{i}_{\cts}(\GL_n(K), R)=\begin{cases}
	R  &  \text{if $i=0,1$}\\
	0  &  \text{otherwise}
	\end{cases}
\end{equation}
if $R$  is any characteristic 0 ring  or  $R=\Z/m\Z$ for suitably chosen $m$, where $R$ is equipped with the discrete topology in both cases.  However, as we pointed out in Section \ref{subsection-3.2}, the module $\Ql$ here arises from the $\ell$-adic pro-\'etale cohomology of Drinfeld spaces, so it is endowed with a natural $\ell$-adic topology as justified by the condensed formalism. Therefore, the result (\ref{eq-classical}) above does not apply directly to our setting; nevertheless, the following analogous statement still holds:

\begin{lem}\label{lem-GL-Ql}
	Let $\Ql$ be the trivial 1-dimensional $\GL_n(K)$-representation equipped with the $\ell$-adic topology. Then 
	\[
		H^{i}_{\cts}(\GL_n(K),\Ql)=\begin{cases}
		\Ql  &  \text{if $i=0,1$}\\
		0  &  \text{otherwise.}
		\end{cases}
	\] 
\end{lem}
\begin{proof}
	We first compute the group cohomology of $\SL_n(K)$ over $\Ql$. Denote $\SL_n(K)$ by $G$. Following notations from \cite[\S II.1]{SS1997-1}, we denote by $X$ the Bruhat-Tits building of $G$, by $X_i$ (resp. $X_{(i)}$)  the set of $i$-dimensional facets (resp. oriented facets) in $X$, by $R_i$ the finite set of $G$-orbit representatives in $X_i$. For a facet $F\in X_i$, let $P^{\dagger}_{F}$ (resp. $P_F$) be the $G$-stabilizer (resp. pointwise $G$-stabilizer) of $F$ and let $\varepsilon _{F}: P^{\dagger}_{F}\rightarrow \{\pm 1\}$ be the orientation character as defined in \cite[\S III.4]{SS1997-1}, which is trivial when restricted to $P_{F}\subseteq P^{\dagger}_{F}$. By \cite[Theorem~II.3.1]{SS1997-1} and \cite[proof of Proposition~III.4.1]{SS1997-1}, we have an exact sequence 
	\[
	0\rightarrow \bigoplus^{}_{F\in R_{n-1}}\cInd^{G}_{P^{\dagger}_{F}}\varepsilon_F\rightarrow \cdots \rightarrow \bigoplus^{}_{F\in R_0} \cInd^{G}_{P^{\dagger}_{F}}\varepsilon_F\rightarrow \Z\rightarrow 0
	\] 
	of smooth representations of $G$; see also \cite[eq.~(4)]{Paulina}. Tensoring this with  $\Ql$ then gives an exact resolution of the representation $\Ql$: 
	\begin{equation}\label{eq-res-Ql}
	0\rightarrow \bigoplus^{}_{F\in R_{n-1}}\cInd^{G}_{P^{\dagger}_{F}}\varepsilon_F\rightarrow \cdots \rightarrow \bigoplus^{}_{F\in R_0} \cInd^{G}_{P^{\dagger}_{F}}\varepsilon_F\rightarrow \Ql\rightarrow 0. 
	\end{equation}
	Applying the continuous cochain functor $C^{\bullet}(G, -)$ to (\ref{eq-res-Ql}) then gives us a double complex. By Shapiro's lemma, the spectral sequence associated to such a double complex is given by 
	\begin{equation}\label{eq-S1}
		E^{i,j}_{1}=\bigoplus^{}_{F\in R_i}H^{j}_{\cts}(P^{\dagger}_{F}, \varepsilon_F)\Rightarrow H^{i+j}_{\cts}(G, \Ql). 
	\end{equation}
	By \cite[\S III.4]{SS1997-1}, since our group $G$ has anisotropic center, the stabilizer group $P^{\dagger}_{F}$ is compact and $P_F\subseteq P^{\dagger}_{F}$ has finite index. Then  \cite[\S II.1]{SS1997-1} further implies that for each facet $F$ we can find  an exact sequence $1\rightarrow U_F\rightarrow P^{\dagger}_{F}\rightarrow I_F\rightarrow 1$ with $U_F\subseteq P_F\subseteq P^{\dagger}_{F}$ pro-$p$ and $I_F$ finite.  Since $\varepsilon|_{P_F}$ is trivial, the same argument as in the proof of Lemma \ref{lem-int-l} implies that $H^{j}_{\cts}(P^{\dagger}_{F}, \Ql)=\Ql$ when $j=0$ and vanishes when $j>0$. Thus the spectral sequence (\ref{eq-S1}) collapses to the 0th row with $E^{i,0}_{1}=\bigoplus^{}_{F\in R_i}\Ql$. But notice that $(E^{\bullet, 0}_{1}, d_1)$ is exactly the augmented cellular cochain complex of $X$ with coefficient in $\Ql$. Since $X$ is contractible, we have $E^{0,0}_{2}=\Ql$ and $E^{i,0}_{2}=0$ for $i>0$. This implies  $H^{i}_{\cts}(G, \Ql)=\Ql$ when $i=0$ and vanishes when $i>0$. 

	Now consider the short exact sequence $1\rightarrow \SL_n(K)\rightarrow \GL_n(K)\xrightarrow{\det}K^{\times}\rightarrow 1$. Using its associated spectral sequence and our computation of $H^{\ast}_{\cts}(\SL_n(K), \Ql)$ above, we get $H^{\ast}_{\cts}(\GL_n(K), \Ql)\cong H^{\ast}_{\cts}(K^{\times}, \Ql)$. Write $K^{\times}\cong \pi^{\Z}\times \OO_K^{\times}$, then Lemma \ref{lem-int-l} further implies $H^{\ast}_{\cts}(K^{\times}, \Ql)\cong H^{\ast}_{\cts}(\Z, \Ql)$, which is isomorphic to $\Ql$ in degree 0 and 1 and vanishes in higher degrees.  
\end{proof}

\begin{thm}\label{thm-l-qiso}
We have isomorphisms 
\[
	H^{r}_{\proet}([\cH^{n-1}_{K}/\GL_n(K)], \Ql)\cong \begin{cases}
	\Ql  &  \text{if $r=0,2$}\\
	\Ql^2  &  \text{if $r=1$}\\
	0  &  \text{otherwise.}
	\end{cases}
\] 
\end{thm}
\begin{proof}
Using the spectral sequence 
\[
	E^{i,j}_{2}=H^{i}_{\cts}(\GL_n(K), H^{j}_{\proet}(\cH^{n-1}_{K}, \Ql))\Rightarrow H^{i+j}_{\proet}([\cH^{n-1}_{K}/\GL_n(K)],\Ql),
\] 
	the theorem follows directly from Proposition \ref{prop-l-descent} and Lemma \ref{lem-GL-Ql}. 
\end{proof}

\subsection{\texorpdfstring{$p$-adic case}{p-adic case}}\label{subsection-5.2}
Same as in Section \ref{subsection-4.2}, carrying out computation directly from the Drinfeld side would lead us to the hurdle of determining 
$
	H^{\ast}_{\cts}(\GL_n(K), \Sp_r(\Qp)^{\ast})
$ 
for  $r$ in between 0 and $n-1$. The representations $\Sp_r(\Qp)^{\ast}$ are  Fr\'echet spaces over $\Qp$, and we will see later in Section \ref{section-6} that even computing such continuous cohomologies in the easiest case when $r=0$ requires some effort. 

Thus we again switch to the Lubin-Tate side and compute $H^{\ast}_{\proet}([\cH^{n-1}_{K}/\GL_n(K)],\Qp)$ through computing $H^{\ast}_{\proet}([\P^{n-1}_{\breve{K}}/\G^{0}_{n}],\Qp)$, as suggested by Corollary \ref{cor-iso}. Let $H$ be the group constructed in the proof of Theorem \ref{thm-p-iso}. Then by construction of $H$, we have an isomorphism 
\begin{equation}\label{equation-iso}
	[\P^{n-1}_{\breve{K}}/\G_n]\cong [\P^{n-1}_{C}/H].
\end{equation}
We first factorize the cohomology of $[\P^{n-1}_{C}/H]$. 
\begin{lem}\label{lem-fib1}
	The $p$-adic pro-\'etale cohomology of the stack $[\P^{n-1}_{C}/H]$ can be expressed as 
	\[
		H^{\ast}_{\proet}([\P^{n-1}_{C}/H],\Qp)\cong H^{\ast}_{\proet}([\P^{n-1}_{C}/\cG_K],\Qp)\otimes H^{\ast}_{\cts}(\OO^{\times}_D, \Qp).
	\] 
\end{lem}
\begin{proof}
	We start by describing the cohomology of $[\P^{n-1}_{C}/\OO_D^{\times}]$. As the cohomology classes of $\P^{n-1}_{C}$ are hyperplanes, $\OO^{\times}_{D}$ acts trivially on $H^{\ast}_{\proet}(\P^{n-1}_{C}, \Qp)$ thus  the terms in the Hochschild-Serre spectral sequence converging to $H^{\ast}_{\proet}([\P^{n-1}_{C}/\OO^{\times}_{D}], \Qp)$ are given by $E^{i,j}_{2}=H^{i}_{\cts}(\OO_D^{\times}, \Qp)\otimes H^{j}_{\proet}(\P^{n-1}_{C}, \Qp)$. Let $\eta\in E^{0,2}_{2}$ be the hyperplane class. Since only the even rows in $E^{i,j}_{2}$ are nonzero, $\eta$ is a permanent cycle. As the spectral sequence is multiplicative and its differentials are derivations, $d_r(\eta)=0$ for all $r\geq 2$ implies $d_r(\eta^{m})=0$ for all $r\geq 2$ and all powers $\eta^{m}$.  Since  powers of $\eta$ generate $H^{\ast}_{\proet}(\P^{n-1}_{C}, \Qp)$, the spectral sequence degenerates on the $E_2$-page and gives $H^{\ast}_{\proet}([\P^{n-1}_{C}/\OO^{\times}_{D}], \Qp)\cong H^{\ast}_{\proet}(\P^{n-1}_{C}, \Qp)\otimes H^{\ast}_{\cts}(\OO^{\times}_{D}, \Qp)$. 

	Using such a tensor product decomposition,  consider the spectral sequence corresponding to further quotienting $[\P^{n-1}_{C}/\OO^{\times}_{D}]$ by $\cG_K$
	\begin{equation}\label{eq-conj}
		E^{i,j}_{2}=H^{i}_{\cts}(\cG_K, \bigoplus^{}_{a+b=j}H^{a}_{\proet}(\P^{n-1}_{C}, \Qp)\otimes H^{b}_{\cts}(\OO^{\times}_{D}, \Qp))\Rightarrow H^{i+j}_{\proet}([\P^{n-1}_{C}/H], \Qp),
	\end{equation} 
	where $\cG_K$ acts on the factor $H^{a}_{\proet}(\P^{n-1}_{C}, \Qp)$ through  Tate twist and acts on the factor $H^{b}_{\cts}(\OO^{\times}_{D}, \Qp)$ through conjugation. By Theorem \ref{thm-projective} and Lemma \ref{lem-p-Gal}, we see that $E^{i,j}_{2}=0$ for $i\geq 2$. Thus   (\ref{eq-conj}) degenerates on the $E_2$-page and gives $H^{\ast}_{\proet}([\P^{n-1}_{C}/H], \Qp)\cong H^{\ast}_{\proet}([\P^{n-1}_{C}/\cG_K], \Qp)\otimes H^{\ast}_{\cts}(\OO^{\times}_{D}, \Qp)^{\cG_K}$. By construction of the group $H$, the conjugation action of $\cG_K$ on $\OO^{\times}_{D}$ factors through the quotient $\hat{\Z}$, where $1\in \hat{\Z}$ acts on $\OO^{\times}_{D}$ by $\Ad(\varpi)$. In \cite[Proposition~3.8.1]{BSSW-1}, it is shown that the conjugation action of $D^{\times}$ on $H^{\ast}_{\cts}(\OO^{\times}_{D}, \Qp)$ is trivial if $D$ is defined over $\Qp$, and the argument there naturally extends to the case when $D$ is defined over a finite extension $K/\Qp$. Thus we have  $H^{\ast}_{\cts}(\OO^{\times}_{D}, \Qp)^{\cG_K}=H^{\ast}_{\cts}(\OO^{\times}_{D}, \Qp)$, which gives the result. 
\end{proof}

\begin{thm}\label{thm-p-qiso}
	Suppose $K$ is a finite extension of $\Qp$ of degree $d$. There is an isomorphism of graded $\Qp$-vector spaces 
	\[
		H^{\ast}_{\proet}([\P^{n-1}_{\breve{K}}/\G_n^{0}],\Qp)\cong H^{\ast}_{\proet}(\P^{n-1}_{K},\Qp)\otimes \Lambda _{\Qp}(x_1,x_3,\cdots,x_{2n-1})^{\otimes d}\otimes \Lambda _{\Qp}(y),
	\] 
	where $\left|x_i\right|=i$, $\left|y\right|=1$, and 
	\[
		H^{\ast}_{\proet}(\P^{n-1}_{K},\Qp)=\begin{cases}
		\Qp  &  \text{if $\ast=0$}\\
		\Qp^{d+1}  &  \text{if $\ast=1$}\\
		\Qp^{d}  &  \text{if $3\leq \ast\leq 2n-1$ and $\ast$ is odd}\\
			0 & \text{otherwise.}
		\end{cases}
	\] 
\end{thm}
\begin{proof}
	Recall from Section \ref{section-2} that $\G^{0}_{n}\cong \G_n\times \Z$.  Since cohomology classes of  $\P^{n-1}_{\breve{K}}$ are hyperplanes, $\Z$ acts trivially on $H ^{\ast}_{\proet}(\P^{n-1}_{\breve{K}}, \Qp)$. As the actions of $\Z$ and $\G_n$ on $\P^{n-1}_{\breve{K}}$ are also commutative, the spectral sequence for the quotient $\P^{n-1}_{\breve{K}}\rightarrow [\P^{n-1}_{\breve{K}}/\G_n]$ is $\Z$-equivariant. Thus $\Z$ also acts trivially on $H^{\ast}_{\proet}([\P^{n-1}_{\breve{K}}/\G_n], \Qp)$. By direct computation, the continuous group cohomology of $\Z$ over the trivial module $\Qp$ is given by $H^{\ast}_{\cts}(\Z, \Qp)\cong \Lambda _{\Qp}(y)$ with $\left|y\right|=1$. This implies  only the 0th and the 1st columns of the  spectral sequence $E^{i,j}_{2}=H^{i}_{\cts}(\Z, H^{j}_{\proet}([\P^{n-1}_{\breve{K}}/\G_n], \Qp))\Rightarrow H^{i+j}_{\proet}([\P^{n-1}_{\breve{K}}/\G^{0}_{n}], \Qp)$ could be nonzero. Thus it degenerates on the $E_2$-page and gives the factorization $H^{\ast}_{\proet}([\P^{n-1}_{\breve{K}}/\G^{0}_{n}], \Qp)\cong H^{\ast}_{\proet}([\P^{n-1}_{\breve{K}}/\G_n], \Qp)\otimes \Lambda _{\Qp}(y)$. Using the isomorphism in  (\ref{equation-iso}) together with Lemma \ref{lem-fib1} and Proposition \ref{prop-generalization}, we get 
	\[
		H^{\ast}_{\proet}([\P^{n-1}_{\breve{K}}/\G_n],\Qp)\cong H^{\ast}_{\proet}(\P^{n-1}_{K}, \Qp)\otimes \Lambda _{\Qp}(x_1,x_3,\cdots,x_{2n-1})^{\otimes d}.
	\] 
The computation of $H^{\ast}_{\proet}(\P^{n-1}_{K},\Qp)$ is immediate from Theorem \ref{thm-projective} and Lemma \ref{lem-p-Gal}. 
\end{proof}

\begin{rem}
	If we try to recover Theorem \ref{thm-l-qiso} via passing to the Lubin-Tate side and computing the cohomology groups  $H^{\ast}_{\proet}([\P^{n-1}_{\breve{K}}/\G^{0}_{n}],\Ql)$, the steps are much simpler than what we did here in the $p$-adic case since all of the spectral sequences are readily degenerating. 
\end{rem}

\section{\texorpdfstring{Application: group cohomology for $\GL_2(\Qp)$}{Application}}\label{section-6}
As we stated in Section \ref{subsection-5.2}, the problem of computing $H^{\ast}_{\cts}(\GL_n(K),\Sp_r(\Qp)^{\ast})$ appears to be challenging. As we will later see in Proposition \ref{prop-GLQp}, even when $r=0$ (so that $\Sp_r(\Qp)^{\ast}\cong \Qp$) and $K=\Qp$, the groups $H^{\ast}_{\cts}(\GL_n(K),\Qp)$ already differ significantly from the discrete case in (\ref{eq-classical}) and the $\ell$-adic case in Lemma \ref{lem-GL-Ql}. But now since Theorem \ref{thm-p-qiso} gives us a description of the  $p$-adic pro-\'etale cohomology of $[\cH^{n-1}_{K}/\GL_n(K)]$, it is natural to ask whether it is possible to deduce some information about the continuous cohomology groups $H^{\ast}_{\cts}(\GL_n(K),\Sp_r(\Qp)^{\ast})$ for $0\leq r\leq n-1$ in return using the spectral sequence
\begin{equation}\label{equation-4}
\begin{split}
	E^{i,j}_{2} &=H^{i}_{\cts}(\GL_n(K), H^{j}_{\proet}(\cH^{n-1}_{\Qp},\Qp))\\
	&= H^{i}_{\cts}(\GL_n(K), \Sp_j(\Qp)^{\ast}) \Rightarrow H^{i+j}_{\proet}([\cH^{n-1}_{K}/\GL_n(K)],\Qp). 
\end{split}
\end{equation}

For simplicity, we will focus on the case when $K=\Qp$. We start with  the case when $r=0$, and $\Sp_0(\Qp)^{\ast}$ is nothing but the trivial representation of $\GL_n(\Qp)$ over $\Qp$. The following lemma is an  immediate consequence of a theorem of Casselman-Wigner generalizing Lazard's comparison theorem between group cohomology and Lie algebra cohomology. 

\begin{lem}\label{lem-SLQp}
	Let $G$ be $\SL_n(\Qp)$ or $\PGL_n(\Qp)$ and let $\Qp$ be the trivial 1-dimensional $G$-representation equipped with the $p$-adic topology. Then we have
	\[
		H^{\ast}_{\cts}(G,\Qp)\cong \Lambda _{\Qp}(x_3,x_5,\cdots,x_{2n-1})
	\] 
	as graded $\Qp$-algebras where $\left|x_i\right|=i$ in the exterior algebra on the right-hand side.
\end{lem}
\begin{proof}
	First notice that $G$ is the group of $\Qp$-points of a connected semisimple group defined over $\Qp$ with $\Lie G=\sl_n(\Qp)$. Then by \cite[Theorem~1]{CW1974-1}, we have an isomorphism of cohomology rings
	$
		H^{\ast}_{\cts}(G,\Qp)\cong H^{\ast}_{\Lie}(\sl_n(\Qp),\Qp),
	$ 
	where the latter one is further isomorphic to $H^{\ast}_{\Lie}(\sl_n(\Q),\Q)\otimes \Qp$. Now as $\SU(n)$ is a compact form of $\SL_n(\R)$, the Lie algbera cohomology of $\sl_n(\Q)$ is isomorphic to the rational de Rham cohomology of $\SU(n)$. Thus we get $H^{\ast}_{\Lie}(\sl_n(\Q),\Q)\cong H^{\ast}_{\dR}(\SU(n))\cong \Lambda _{\Q}(x_3,x_5,\cdots,x_{2n-1})$ where the degree of $x_i$ is $i$. 
\end{proof}

Now we compute the continuous group cohomology ring $H^{\ast}_{\cts}(\GL_n(\Qp), \Qp)$. 
\begin{prop}\label{prop-GLQp}
	Let $\Qp$ be the trivial 1-dimensional $\GL_n(\Qp)$-representation equipped with the $p$-adic topology. Then we have 
	\[
		H^{\ast}_{\cts}(\GL_n(\Qp),\Qp)\cong \Lambda _{\Qp}(x,y,x_3,x_5,\cdots,x_{2n-1})
	\] 
	as graded $\Qp$-algebras where $\left|x\right|=\left|y\right|=1$ and $\left|x_i\right|=i$ in the exterior algebra on the right-hand side.  
\end{prop}
\begin{proof}
	We first compute the continuous cohomology of  the center $\Qp^{\times}$ of $\GL_n(\Qp)$.  Consider the decomposition $\Qp^{\times}\cong p^{\Z}\times \Zp^{\times}$. By Lemma \ref{lem-Lazard}, we know $H^{\ast}_{\cts}(\Zp^{\times},\Qp)\cong \Lambda _{\Qp}(x)$ with $\left|x\right|=1$. For the other part, we have  $H^{0}_{\cts}(p^{\Z},\Qp)=\Qp$ and $H^{1}_{\cts}(p^{\Z},\Qp)$ is isomorphic to the  group of continuous homomorphisms from $\Z$ to $\Qp$, which is again $\Qp$. As all higher cohomology groups vanish, we have $H^{\ast}_{\cts}(p^{\Z},\Qp)\cong \Lambda _{\Qp}(y)$ with $\left|y\right|=1$. Thus by K\"unneth formula, we have $H^{\ast}_{\cts}(\Qp^{\times},\Qp)\cong \Lambda _{\Qp}(x,y)$. 

	Next, since we have the short exact sequence 
	\begin{equation}\label{eq-split}
		1\rightarrow \SL_n(\Qp)\rightarrow \GL_n(\Qp)\xrightarrow{\det} \Qp^{\times}\rightarrow 1,
	\end{equation}
	we want to access the cohomology ring of $\GL_n(\Qp)$ using the Hochschild-Serre spectral sequence 
	\begin{equation}\label{eq-GLHS}
		E^{i,j}_{2}=H^{i}_{\cts}(\Qp^{\times},H^{j}_{\cts}(\SL_n(\Qp),\Qp))\Rightarrow H^{i+j}_{\cts}(\GL_n(\Qp),\Qp).
	\end{equation}
	By Lemma \ref{lem-SLQp} and our computation of $H^{\ast}_{\cts}(\Qp^{\times},\Qp)$ above, one immediately notices that $E_2^{i,j}=0$ when $i\geq 3$. Since (\ref{eq-split}) splits, the differential maps $d^{0,j}_{2}$ on the $E_2$-page are also 0 for any $j\geq 0$. Thus (\ref{eq-GLHS}) degenerates on the $E_2$-page and we get $H^{\ast}_{\cts}(\GL_n(\Qp),\Qp)\cong H^{\ast}_{\cts}(\SL_n(\Qp),\Qp)\otimes H^{\ast}_{\cts}(\Qp^{\times},\Qp)$. 
\end{proof}

The following lemma gives a  factorization of $H^{\ast}_{\proet}([\cH^{n-1}_{\Qp}/\GL_n(\Qp)],\Qp)$ that we will need later. 
\begin{lem}\label{lem-fib2}
	The $p$-adic pro-\'etale cohomology of the stack $[\cH^{n-1}_{\Qp}/\GL_n(\Qp)]$ can be expressed as 
	\[
		H^{\ast}_{\proet}([\cH^{n-1}_{\Qp}/\GL_n(\Qp)],\Qp) \cong H^{\ast}_{\proet}([\cH^{n-1}_{\Qp}/\PGL_n(\Qp)], \Qp)\otimes H^{\ast}_{\cts}(\Qp^{\times},\Qp).
	\] 
\end{lem}
\begin{proof}
	Consider the short exact sequence $1\rightarrow \Qp^{\times}\rightarrow \GL_n(\Qp)\rightarrow \PGL_n(\Qp)\rightarrow 1$. From the definition of $\cH^{n-1}_{\Qp}$, we see that $\Qp^{\times}$ acts trivially on $\cH^{n-1}_{\Qp}$. Thus we have  a commutative diagram of stacks:
\[
\begin{tikzcd}
	B\Qp^{\times} \arrow[r, hook, "i_1"] & {[\cH^{n-1}_{\Qp}/\GL_n(\Qp)]} \arrow[r, "\pi_1"]\arrow[d, "f_1"] & {[\cH^{n-1}_{\Qp}/\PGL_n(\Qp)]}\arrow[d, "f_2"]\\
	B\Qp^{\times} \arrow[r, hook, "i_2"] \arrow[u, equal] & B\GL_n(\Qp) \arrow[r, "\pi_2"] & B\PGL_n(\Qp)
\end{tikzcd}
\]
where each row of this diagram gives a fibration of stacks. Specifically, the Leray-Serre spectral sequence corresponding to the fibration described by the first row gives
\[
	E^{i,j}_{2}=H^{i}_{\proet}([\cH^{n-1}_{\Qp}/\PGL_n(\Qp)], R^{j}i_{1,\ast}\underline{\Qp}) \Rightarrow H^{i+j}_{\proet}([\cH^{n-1}_{\Qp}/\GL_n(\Qp)],\Qp)
\] 
	where $R^{j}i_{1,\ast}\underline{\Qp}$ is the $\Qp$-local system associated to the cohomology of the fiber $B\Qp^{\times}$.  Since the conjugation action of $\PGL_n(\Qp)$ on $\Qp^{\times}$ is trivial, the monodromy action is trivial and the local system is in fact constant. Thus the spectral sequence above can be rewritten into 
	\begin{equation}\label{equation-LS2}
		E^{i,j}_{2}=H^{i}_{\proet}([\cH^{n-1}_{\Qp}/\PGL_n(\Qp)], H^{j}_{\cts}(\Qp^{\times},\Qp)) \Rightarrow H^{i+j}_{\proet}([\cH^{n-1}_{\Qp}/\GL_n(\Qp)],\Qp)
	\end{equation}
	and we are left to show this spectral sequence degenerates on the $E_2$-page. 

	By Lemma \ref{lem-SLQp} and Proposition \ref{prop-GLQp} and  dimension-counting,  we see the Leray-Serre spectral sequence associated to the second row of the diagram above degenerates. So the edge homomorphism
	\[
		i_2^{\ast}: H^{j}_{\cts}(\GL_n(\Qp),\Qp)\rightarrow H^{j}_{\cts}(\Qp^{\times},\Qp) 
	\] 
	is surjective for all $j\geq 0$. By commutativity of the diagram, we have $i_2^{\ast}=i_1^{\ast}\circ f_1^{\ast}$. Thus the edge homomorphism of the spectral sequence (\ref{equation-LS2})
	\[
		i_1^{\ast}: H^{j}_{\proet}([\cH^{n-1}_{\Qp}/\GL_n(\Qp)],\Qp)\rightarrow E^{0,j}_{\infty}\subseteq E^{0,j}_{2}=H^{j}_{\cts}(\Qp^{\times}, \Qp)
	\] 
	is surjective for all $j\geq 0$ and (\ref{equation-LS2}) degenerates on the second page. 
\end{proof}

In the rest of this section, we restrict to the case when $n=2$ and use Theorem \ref {thm-p-qiso} together with Proposition \ref{prop-GLQp} to deduce a description of $H^{\ast}_{\cts}(\GL_2(\Qp), \Sp_1(\Qp)^{\ast})$. Notice that when $n=2$, the only $\Sp_r(\Qp)^{\ast}$ appearing are for $r=0$ and $r=1$.

\begin{lem}\label{lem-PGL-deg}
The Hochschild-Serre spectral sequence 
\[
	E^{i,j}_{2}=H^{i}_{\cts}(\PGL_2(\Qp), H^{j}_{\proet}(\cH^{1}_{\Qp}, \Qp))\Rightarrow H^{i+j}_{\proet}([\cH^{1}_{\Qp}/\PGL_2(\Qp)], \Qp)
\] 
degenerates on the $E_2$-page. 
\end{lem}
\begin{proof}
	By Proposition \ref{prop-p-descent}, only the first three rows of this spectral sequence are nonzero. By Lemma \ref{lem-SLQp}, we have $H^{\ast}_{\cts}(\PGL_2(\Qp), \Qp)\cong \Lambda _{\Qp}(x_3)$, and we also write $H^{i}$ for $H^{i}_{\cts}(\PGL_2(\Qp), \Sp_1(\Qp)^{\ast})$.
	As many terms on the $E_2$-page are zero, almost all $E^{i,j}_{2}$ are already stable  with only two differentials we need to check more carefully. The first one is the map 
	$
	d_3^{0,2}: H^{0}\rightarrow E_3^{3,0}=\Qp.
	$ 
	But by definition of the Steinberg representation $\Sp_1(\Qp)$, there are no stable vectors in $\Sp_1(\Qp)^{\ast}$ under the action of $\PGL_2(\Qp)$, and thus $d^{0,2}_{3}$ is zero. The second map we need to check is 
	$
	d^{1,2}_{2}: H^{1}\rightarrow E^{3,1}_{2}=\Qp^2.
	$ 
	By Proposition \ref{prop-p-descent}, we have 
$
	H^{1}_{\proet}(\cH^{1}_{\Qp},\Qp)\cong \Qp^2
$ 
 	and the term  $E^{3,1}_{2}$ here arises from 
$
	H^{3}_{\cts}(\PGL_2(\Qp),\Qp^2)\cong (H^{3}_{\cts}(\PGL_2(\Qp), \Qp))^{2}.
$
	Now the stability of  $E^{3,0}_{2}$ implies that every  element of  $H^{3}_{\cts}(\PGL_2(\Qp),\Qp)$ is a permanent cycle for the spectral sequence. So the elements in $E^{3,1}_{2}$ are also permanent cycles, and $d^{1,2}_{2}$ is zero. 
\end{proof}

\begin{thm}\label{thm-Steinberg-coh}
	For the dual Steinberg representation $\Sp_1(\Qp)^{\ast}$ of $\GL_2(\Qp)$, we have 
	\[
		H^{\ast}_{\cts}(\PGL_2(\Qp), \Sp_1(\Qp)^{\ast})\cong H^{\ast}_{\cts}(\PGL_2(\Qp), \Qp)[-1]
	\]
	and also
	\[
		H^{\ast}_{\cts}(\GL_2(\Qp),\Sp_1(\Qp)^{\ast})\cong H^{\ast}_{\cts}(\GL_2(\Qp),\Qp)[-1].
	\]
\end{thm}
\begin{proof}
	Following notations from Lemma \ref{lem-PGL-deg}, we write $H^{i}$ for $H^{i}_{\cts}(\PGL_2(\Qp), \Sp_1(\Qp)^{\ast})$ and also denote the dimension of $H^{i}$ (as a $\Qp$-vector space) by $h^{i}$. Then using Theorem \ref{thm-p-qiso} and Lemma \ref{lem-fib2}, we get an explicit description of the dimension of the $p$-adic pro-\'etale cohomology group of $[\cH^{1}_{\Qp}/\PGL_2(\Qp)]$ in all degrees. The degeneration of the spectral sequence in Lemma \ref{lem-PGL-deg} then allows us to determine the dimensions of $H^{i}$, yielding 
	\[
	h^{0}=0,\quad h^1=1, \quad h^{2}=0, \quad h^{3}=0, \quad h^4=1,
	\] 
	and $h^i=0$ for $i\geq 5$. Comparing this with the description of the cohomology ring of $\PGL_2(\Qp)$ over the trivial representation $\Qp$ given by Lemma \ref{lem-SLQp} gives the claimed shifting. 

	For the case of $\GL_2(\Qp)$, we first denote by $\wt{h}^{i}$ the dimension of the $i$-th continuous cohomology group of $\GL_2(\Qp)$ over $\Sp_1(\Qp)^{\ast}$. 
	Then consider the spectral sequence 
	\[
		E^{i,j}_{2}=H^{i}_{\cts}(\PGL_2(\Qp), H^{j}_{\cts}(\Qp^{\times},\Sp_1(\Qp)^{\ast}))\Rightarrow H^{i+j}_{\cts}(\GL_2(\Qp), \Sp_1(\Qp)^{\ast}).
	\] 
	Since the center of $\GL_2(\Qp)$ acts trivially on  $\Sp_1(\Qp)^{\ast}$ by definition, this spectral sequence can be written as 
	\begin{equation}\label{equation-GL2-deg}
		E^{i,j}_{2}=H^{i}_{\cts}(\PGL_2(\Qp), H^{j}_{\cts}(\Qp^{\times},\Qp)\otimes \Sp_1(\Qp)^{\ast})\Rightarrow H^{i+j}_{\cts}(\GL_2(\Qp), \Sp_1(\Qp)^{\ast}). 
	\end{equation}
	Proposition \ref{prop-GLQp}, together with our computation of the $\PGL_2(\Qp)$ case above, gives all the terms $E^{i,j}_{2}$. Specifically, only the first three entries of the first and the fourth columns of this spectral sequence are nonzero. Therefore, all of the differential maps on the $E_2$-page are zero except possibly
	$
	d^{1,2}_{3}: E^{1,2}_{3}\rightarrow E^{4,0}_{3}.
	$ 
	By direct computation, both  $E^{1,2}_{3}$ and $E^{4,0}_{3}$ are one-dimensional. Thus if $d^{1,2}_{3}$ is nonzero, we would have $E^{4,0}_{4}=E^{4,0}_{\infty}=0$, which further gives $\wt{h}^{4}=0$. But this is impossible as we already have $h^{4}=1$. Therefore, the spectral sequence (\ref{equation-GL2-deg}) degenerates on the $E_2$-page which then gives
	\[
		\wt{h}^{0}=0, \quad \wt{h}^{1}=1,\quad \wt{h}^{2}=2, \quad \wt{h}^{3}=1,\quad \wt{h}^{4}=1,\quad \wt{h}^{5}=2,\quad \wt{h}^{6}=1,
	\] 
	and $\wt{h}^{i}=0$ for $i\geq 7$. Comparing this with Proposition \ref{prop-GLQp} finishes the proof.
\end{proof}

When $n>2$, one cannot deduce descriptions of  $H^{\ast}_{\cts}(\GL_n(\Qp), \Sp_r(\Qp)^{\ast})$ in the same way  as when $n=2$: on one hand there would be  too many $r$ we need to deal with simultaneously; on the other hand, we also do not know whether the spectral sequence (\ref{equation-4}) degenerates on the $E_2$-page. 

\begin{rem}\label{rem-evidence1}
	Nevertheless, Theorem \ref{thm-Steinberg-coh} still offers a plausible indication for the case of general $n$, suggesting that $H^{\ast}_{\cts}(\GL_n(\Qp), \Sp_r(\Qp)^{\ast})$ might be given by a degree-$r$ shift of the cohomology ring  $H^{\ast}_{\cts}(\GL_n(\Qp), \Qp)$. Indeed, under the supposition that  such a "shifting" holds,  if  we further assume (\ref{equation-4}) degenerates on the $E_2$-page, then inserting these cohomology groups into  (\ref{equation-4}) reproduces the cohomology groups of $[\cH^{n-1}_{\Qp}/\GL_n(\Qp)]$ given in Theorem \ref{thm-p-qiso}. 
\end{rem}

\begin{rem}\label{rem-evidence2}
Another reason we find such a speculation of "shifting" reasonable is its similarity to the case of continuous cohomology of $\GL_n(\Qp)$ with respect to duals of generalized Steinberg representations  over $\C$. From  \cite[Proposition~X.4.7]{BW} and taking into consideration the center of $\GL_n(\Qp)$, one could  deduce 
\[
	H^{i}_{\cts}(\GL_n(\Qp), \Sp_r(\C)^{\ast})\cong \begin{cases}
	\C  &  \text{if $i=r,r+1$}\\
	0  &  \text{otherwise.}
	\end{cases}
\] 
In other words, the cohomology ring $H^{\ast}_{\cts}(\GL_n(\Qp),\Sp_r(\C)^{\ast})$ is exactly a degree-$r$ shift of $H^{\ast}_{\cts}(\GL_n(\Qp), \C)$.
\end{rem}

In the following section, we will prove that such a speculation is indeed true, employing the language of condensed math from \cite{CS2019-1} and also  the framework of solid representations recently developed by Rodrigues Jacinto--Rodr\'iguez Camargo in \cite{RJRC2022-1} and  \cite{RJRC2025-1}.

\section{\texorpdfstring{Generalization to $\GL_n(\Qp)$}{Generalization to GL_n(Qp)}}
Our goal of this section is to  generalize  Theorem \ref{thm-Steinberg-coh}  to $\GL_n(\Qp)$. Specifically, we will show:  

\begin{thm}\label{thm-shift}
For $0\leq r\leq n-1$, we have an isomorphism of graded $\Qp$-vector spaces 
\[
	H^{\ast}_{\cts}(\GL_n(\Qp), \Sp_r(\Qp)^{\ast})\cong \Lambda_{\Qp}(x,y,x_3,x_5,\cdots, x_{2n-1})[-r],
\] 
where $\left|x\right|=\left|y\right|=1$ and $\left|x_i\right|=i$ in the exterior algebra on the right-hand side. 
\end{thm}

For a discrete ring $R$, the cohomology groups $H^{\ast}_{\cts}(\GL_n(\Qp), \Sp_r(R)^{\ast})$ are well-understood if $R$ has characteristic 0 or $R=\Z/m\Z$ for suitably chosen $m$ by the work of \cite{Orlik2005-1} and \cite{Dat}. For such a ring $R$, \cite[Corollary~2]{Orlik2005-1} and \cite[Corollaire~2.1.7]{Dat} give a description of the extension groups 
\begin{equation}\label{eq-ext}
\Ext^{\ast}_{G}(\Sp_I(R), \Sp_J(R))
\end{equation}
for any reductive $p$-adic Lie group $G$ and any subsets $I,J$ of simple roots of $G$. One can then compute $H^{\ast}_{\cts}(\GL_n(\Qp), \Sp_r(R)^{\ast})$ by realizing them as $\Ext^{\ast}_{\GL_n(\Qp)}(\Sp_r(R), \Sp_0(R))$. 

In an attempt to use such a strategy to compute the cohomology of $\GL_n(\Qp)$ over $\Sp_r(\Qp)^{\ast}$, one immediately encounters the following problem: the ring $R$ and also the representations $\Sp_I(R)$ in \cite{Orlik2005-1} and \cite{Dat} are equipped with the discrete topology, and the  $\Ext$-groups in (\ref{eq-ext}) are computed in the abelian category of smooth $G$-representations with the discrete topology; on the other hand,  the representations $\Sp_r(\Qp)$ are equipped with the $p$-adic topology, but  the category of smooth $G$-representations over topological abelian groups is not an abelian category. Thus it no longer makes sense to talk about $\Ext$-groups in such a category. 

To salvage this, we find the theory of solid representations developed in \cite{RJRC2022-1} and \cite{RJRC2025-1} perfect for our purpose as it provides  us some abelian categories which contain representations like  $\Sp_r(\Qp)$. This allows us to study $\Ext$-groups between generalized Steinberg representations of $\GL_n(\Qp)$ over $\Qp$ using strategies analogous to \cite{Orlik2005-1} and \cite{Dat}. Specifically, the key ideas follow closely from those appeared in \cite{Orlik2005-1}.

For the rest of this section, let $G=\GL_n(\Qp)$ and also let $K=\Qp$. Following the notations\footnote{Unlike  in \cite{RJRC2025-1}, we will not use the derived language.} of \cite{RJRC2025-1}, we let $\Rep^{}_{K_{\square}}(G)$ be the category of continuous $G$-representations over solid $K$-vector spaces and let $\Rep^{\infty}_{K_{\square}}(G)$ be its subcategory of smooth $G$-representations. Both of these are abelian categories and thus have enough injectives (upon fixing the cardinal $\kappa$), see e.g. \cite[Remark~2.11]{Bosco2023-1}. The category $\Rep_{K_{\square}}(G)$ also has enough projectives.  We denote by $\Hom_{K_{\square}[G]}$ the Hom-functor of $\Rep_{K_{\square}}(G)$ and we let $\Ext_{K_{\square}[G]}$ denote its derived functors. For $\cV\in \Rep_{\Ksolid}(G)$, we denote its smooth vectors by $\cV^{\infty}$ and its smooth dual $(\cV^{\ast})^{\infty}$ by $\wt{\cV}$.

\begin{lem}
	For any $I\subseteq \Delta $ and parabolic subgroup $P_I\subseteq G$, the representations $\underline{\Ind^{G}_{P_I}\mathbbm{1}}, \underline{\Sp_I(\Qp)}, \underline{\Sp_I(\Qp)^{\ast}}$ are all objects of $\Rep_{\Ksolid}(G)$. 
\end{lem}
\begin{proof}
	The representations $\Ind^{G}_{P_I}\mathbbm{1}, \Sp_I(\Qp), \Sp_I(\Qp)^{\ast}$ are all complete locally convex $G$-representations over $\Qp$.  The lemma then follows from  \cite[Proposition~3.7]{RJRC2022-1}.  
\end{proof}

\begin{lem}\label{lem-inj}
The functor 
\begin{align*}
	(-)^{\infty}: \Rep_{K_{\square}}(G)&\rightarrow \Rep^{\infty}_{K_{\square}}(G)\\
	\cV&\mapsto \cV^{\infty}:=\dirlim_{\substack{H\leq G\\ \text{compact open}}} \cV^{H}
\end{align*}
preserves injective resolutions. 
\end{lem}
\begin{proof}
	For $\cV\in \Rep_{K_{\square}}(G)$ and $H\leq G$ compact open, its clear that the functor $\cV\mapsto \cV^{H}$ is exact. As $\Rep_{K_{\square}}(G)$ is abelian, filtered colimits of exact functors remain exact, which implies exactness of $(-)^{\infty}$. As both categories are abelian and have enough injectives, we deduce from  \cite[\href{https://stacks.math.columbia.edu/tag/015Z}{Tag 015Z}]{stacks-project} that  $(-)^{\infty}$ preserves injective resolutions.
\end{proof}

\begin{prop}\label{prop-to_smooth}
	For $\cU\in \Rep^{\infty}_{K_{\square}}(G)$ and $\cV\in \Rep_{K_{\square}}(G)$, we have $\Ext^{i}_{K_{\square}[G]}(\cU,\cV)\cong \Ext^{i}_{K_{\square}[G]}(\cU,\cV^{\infty})$ for any $i\geq 0$. Specifically, if $V$ is a complete locally convex representation  of $G$ over $K$, then $H^{i}_{\cts}(G, V)\cong H^{i}_{\cts}(G, V^{\infty})$ for any $i\geq 0$.
\end{prop}
\begin{proof}
	Take an injective resolution $0\rightarrow \cV\rightarrow \cI^{\bullet}$ of $\cV$ in $\Rep_{\Ksolid}(G)$ and also consider the injective resolution $0\rightarrow \cV^{\infty}\rightarrow (\cI^{\bullet})^{\infty}$ of $\cV^{\infty}$ in $\Rep^{\infty}_{\Ksolid}(G)$ given by Lemma \ref{lem-inj}. As $\cU$ is smooth, we have $f(\cU)\subseteq \cW^{\infty}$ for any continuous $G$-homomorphism $f: \cU\rightarrow \cW$. This implies $\Hom_{\Ksolid[G]}(\cU, \cI^{\bullet})\cong  \Hom_{\Ksolid[G]}(\cU, (\cI^{\bullet})^{\infty})$ and the first assertion follows. For $V$ a complete locally convex $G$-representation, $V^{\infty}$ is also complete and locally convex.  
The isomorphism $H^{i}_{\cts}(G, V)\cong H^{i}_{\cts}(G, V^{\infty})$ then follows from \cite[Lemma~5.2]{RJRC2022-1} and the first statement.
\end{proof}

\begin{lem}\label{lem-swap}
	Given  $\cU, \cV\in \Rep^{\infty}_{K_{\square}}(G)$, if $\cV$ is further admissible, then
	\[
		\Ext^{i}_{K_{\square}[G]}(\cU,\cV)\cong \Ext^{i}_{K_{\square}[G]}(\wt{\cV}, \wt{\cU})
	\] 
	for any $i\geq 0$. 
\end{lem}
\begin{proof}
	The classical case of such an isomorphism appeared in \cite[Lemma~6]{Orlik2005-1}, and we follow the same strategy here. As $\Rep_{\Ksolid}(G)$ is an abelian category with enough injectives and projectives, we can take a projective resolution $0\leftarrow \cU\leftarrow \cP^{\bullet}$ in $\Rep_{\Ksolid}(G)$ and its smooth dual gives an injective resolution of $\wt{\cU}$. By \cite[Proposition~I.4.13]{Vig1}, we have $\Hom_{\Ksolid[G]}(\cP^{i}, \wt{\cV})\cong \Hom_{\Ksolid[G]}(\cV, \wt{\cP^{i}})$, which implies $\Ext^{i}_{\Ksolid[G]}(\cU, \wt{\cV})\cong \Ext^{i}_{\Ksolid[G]}(\cV, \wt{\cU})$. As $\cV$ is admissible, we have  $\wt{\wt{\cV}}\cong \cV$ by \cite[Proposition~2.1.10]{Casselman1} and the claim follows.
\end{proof}

\begin{lem}\label{lem-ind_vanishing}
	For any parabolic subgroup $P_I\subseteq G$ associated to $I\subsetneq \Delta $, we have $\Ext^{\ast}_{K_{\square}[G]}(\underline{\Ind^{G}_{P_I}\mathbbm{1}}, \underline{\mathbbm{1}})=0$.
\end{lem}
\begin{proof}
	Let $\delta _{P_I}$ denote the modulus character of the parabolic subgroup $P_I$. Then by \cite[Chapitre I,5.11]{Vig1} we have $\wt{\Ind^{G}_{P_I}\mathbbm{1}}=\Ind^{G}_{P_I}\delta _{P_I}$. Thus we get 
	\begin{align*}
		\Ext^{\ast}_{\Ksolid[G]}(\underline{\Ind^{G}_{P_I}\mathbbm{1}}, \underline{\mathbbm{1}}) &\cong \Ext^{\ast}_{\Ksolid[G]}(\underline{\mathbbm{1}}, \underline{\wt{\Ind^{G}_{P_I}\mathbbm{1}}})\\
		&\cong \Ext^{\ast}_{\Ksolid[G]}(\underline{\mathbbm{1}}, \underline{\Ind^{G}_{P_I}\delta _{P_I}})\\
		&\cong \Ext^{\ast}_{\Ksolid[P_I]}(\underline{\mathbbm{1}}, \underline{\delta _{P_I}})
	\end{align*}
	where we used Lemma \ref{lem-swap} in the first isomorphism and the third isomorphism comes from Shapiro's lemma. Since $I$ is a proper subset of $\Delta $, the modulus character $\delta _{P_I}$ is not the trivial character and \cite[Proposition~XI.1.9]{BW} implies there does not exist any  extension of $\mathbbm{1}$ by $\delta _{P_I}$. This gives $\Ext^{\ast}_{\Ksolid[G]}(\underline{\mathbbm{1}}, \underline{\delta _{P_I}})=0$, as desired. 
\end{proof}

\begin{proof}[Proof of Theorem \ref{thm-shift}]
	As the representations $\Sp_r(\Qp)^{\ast}$ are complete and locally convex, Proposition \ref{prop-to_smooth} and Lemma \ref{lem-swap} together give $H^{\ast}_{\cts}(\GL_n(\Qp), \Sp_r(\Qp)^{\ast})\cong \Ext^{\ast}_{\Ksolid[G]}(\underline{\Sp_r(\Qp)}, \underline{\mathbbm{1}})$. Since  $\Ext^{\ast}_{\Ksolid[G]}(\underline{\mathbbm{1}}, \underline{\mathbbm{1}})$ is just $H^{\ast}_{\cts}(\GL_n(\Qp),\Qp)$ which is isomorphic to $\Lambda _{\Qp}(x,y,x_3,x_5,\cdots,x_{2n-1})$ by Proposition \ref{prop-GLQp}, it suffices to show $\Ext^{\ast}_{\Ksolid[G]}(\underline{\Sp_r(\Qp)}, \underline{\mathbbm{1}})\cong \Ext^{\ast}_{\Ksolid[G]}(\underline{\mathbbm{1}},\underline{\mathbbm{1}})[-r]$. 

	Now recall that $\Sp_r(\Qp)$ is defined as the generalized Steinberg representation of $\GL_n(\Qp)$ associated to the subset of simple roots $\{1,\cdots,n-1-r\}\subseteq \Delta $. Let us denote the set $\{1,\cdots, n-1-r\}$ by $I_r$. Then \cite[Proposition~13]{SS1991-1} (see also p.88 of loc. cit.) gives an acyclic resolution of $\Sp_r(\Z)$ 
	\[
		0\rightarrow \Z\rightarrow \bigoplus_{\substack{I_r\subseteq I\subseteq \Delta \\ \left|\Delta \bs I\right|=1}} \Ind^{G}_{P_I}\mathbbm{1}\rightarrow \cdots \rightarrow \bigoplus^{}_{\substack{I_r\subseteq I\subseteq \Delta \\ \left|I\bs I_r\right|=1}}\Ind^{G}_{P_I}\mathbbm{1}\rightarrow \Ind^{G}_{P_{I_r}}\mathbbm{1}\rightarrow \Sp_r(\Z)\rightarrow 0
	\] 
	in the category of smooth $\GL_n(\Qp)$-representations with discrete topology. Tensoring this exact sequence with $\Qp$ and condensing it, we get an acyclic resolution of $\underline{\Sp_r(\Qp)}$ 
	\begin{equation}\label{eq-proj_res}
		0\rightarrow \underline{\Qp}\rightarrow \bigoplus_{\substack{I_r\subseteq I\subseteq \Delta \\ \left|\Delta \bs I\right|=1}} \underline{\Ind^{G}_{P_I}\mathbbm{1}}\rightarrow \cdots \rightarrow \bigoplus^{}_{\substack{I_r\subseteq I\subseteq \Delta \\ \left|I\bs I_r\right|=1}}\underline{\Ind^{G}_{P_I}\mathbbm{1}}\rightarrow \underline{\Ind^{G}_{P_{I_r}}\mathbbm{1}}\rightarrow \underline{\Sp_r(\Qp)}\rightarrow 0
	\end{equation}
	in the category $\Rep^{\infty}_{\Ksolid}(G)$. Now apply the functor $\Hom_{\Ksolid[G]}(-, \underline{\mathbbm{1}})$ to (\ref{eq-proj_res}) and pick an injective resolution $0\rightarrow \underline{\mathbbm{1}}\rightarrow \cI^{\bullet}$ in the category $\Rep^{\infty}_{\Ksolid}(G)$, we get a double complex of the form 
	\[
	\begin{tikzcd}[sep=small]
		& \vdots & \vdots & \\
		\cdots\arrow[r] & \Hom_{\Ksolid[G]}(\bigoplus_{\substack{I_r\subseteq I\subseteq \Delta \\ \left|I\bs I_r\right|=i}}\underline{\Ind^{G}_{P_I}\mathbbm{1}}, \cI^{j+1})\arrow[r]\arrow[u] & \Hom_{\Ksolid[G]}(\bigoplus^{}_{\substack{I_r\subseteq I\subseteq \Delta \\ \left|I\bs I_r\right|=i+1}}\underline{\Ind^{G}_{P_I}\mathbbm{1}}, \cI^{j+1})\arrow[r]\arrow[u] & \cdots\\
		\cdots\arrow[r] & \Hom_{\Ksolid[G]}(\bigoplus^{}_{\substack{I_r\subseteq I\subseteq \Delta \\ \left|I\bs I_r\right|=i}}\underline{\Ind^{G}_{P_I}\mathbbm{1}}, \cI^{j})\arrow[r]\arrow[u] & \Hom_{\Ksolid[G]}(\bigoplus^{}_{\substack{I_r\subseteq I\subseteq \Delta \\ \left|I\bs I_r\right|=i+1}}\underline{\Ind^{G}_{P_I}\mathbbm{1}}, \cI^{j})\arrow[r]\arrow[u] & \cdots\\
		& \vdots\arrow[u] & \vdots\arrow[u] & 
	\end{tikzcd}
	\] 
	The $E_1$-page of the spectral sequence associated to such a double complex then gives 
	\begin{equation}\label{eq-E1_vanish}
		E^{i,j}_{1}=\Ext^{j}_{\Ksolid[G]}(\bigoplus^{}_{\substack{I_r\subseteq I\subseteq \Delta \\ \left|I\bs I_r\right|=i}}\underline{\Ind^{G}_{P_I}\mathbbm{1}}, \underline{\mathbbm{1}})\Rightarrow \Ext^{i+j}_{\Ksolid[G]}(\underline{\Sp_r(\Qp)}, \underline{\mathbbm{1}}).
	\end{equation}
	By Lemma \ref{lem-ind_vanishing}, the columns of (\ref{eq-E1_vanish}) are all zero except the $r$-th column, which is $\Ext^{\ast}_{\Ksolid[G]}(\underline{\mathbbm{1}}, \underline{\mathbbm{1}})$. Thus we get  an isomorphism  $\Ext^{\ast}_{\Ksolid[G]}(\underline{\Sp_r(\Qp)}, \underline{\mathbbm{1}})\cong \Ext^{\ast}_{\Ksolid[G]}(\underline{\mathbbm{1}}, \underline{\mathbbm{1}})[-r]$, and this finishes the proof.  
\end{proof}

\printbibliography
\end{document}